\documentclass[a4paper,10pt]{amsart}

\usepackage{amssymb,latexsym}
\usepackage{amsthm}
\usepackage{amsmath}
\usepackage{amsfonts}
\usepackage{mathrsfs}
\usepackage{mathdots}
\usepackage{graphicx}
\usepackage[all]{xy}
\usepackage{chngcntr}
\usepackage{fancyhdr}
\usepackage[top=1in,bottom=1in,left=1.25in,right=1.25in, marginparwidth=1in,]{geometry}

\usepackage[mathscr]{eucal}

\usepackage{dynkin-diagrams}

\usepackage{tikz-cd}
\usepackage{stmaryrd}

\numberwithin{equation}{subsection}
\newtheorem{theorem}{Theorem}[section]
\newtheorem*{itheorem}{Theorem}
\newtheorem*{iproposition}{Proposition}
\newtheorem{definition}[theorem]{Definition}
\newtheorem{remark}[theorem]{Remark}
\newtheorem{conjecture}[theorem]{Conjecture}
\newtheorem{proposition}[theorem]{Proposition}
\newtheorem{corollary}[theorem]{Corollary}
\newtheorem{lemma}[theorem]{Lemma}
\newtheorem{example}[theorem]{Example}

\def\Q{\mathbb{Q}}
\def\F{\mathbb{F}}

\def\Z{\mathbb{Z}}

\def\E{\mathcal{E}}

\def\Fc{\mathcal{F}}

\def\M{\mathcal{M}}
\def\N{\mathcal{N}}

\def\Bun{\mathrm{Bun}}

\def\Hom{\mathrm{Hom}}
\def\Lie{\mathrm{Lie}}

\def\Mod{\mathrm{Mod}}

\def\Gal{\mathrm{Gal}}

\def\Rep{\mathrm{Rep}}

\def\Spec{\mathrm{Spec}}

\def\deg{\mathrm{deg}}

\def\lra{\longrightarrow}
\def\ra{\rightarrow}

\def\st{\stackrel}
\def\tr{\textrm}

\def\Fbar{\overline{F}}

\def\Phic{\Phi^\vee}

\def\s{\sigma}
\def\Gr{\mathrm{Gr}}

\def\BpdR{\mathrm{B_{dR}^+}}

\def\e{\epsilon}

\begin{document}

\title{Fargues-Rapoport conjecture for $p$-adic period domains in the non-basic case}
\author{Miaofen Chen}
\date{}
\address{School of Mathematical Sciences\\
	Shanghai Key Laboratory of PMMP\\
	East China Normal University\\
	No. 500, Dong Chuan Road\\
	Shanghai 200241, China}\email{mfchen@math.ecnu.edu.cn}
	
\renewcommand\thefootnote{}
\footnotetext{
key words: $p$-adic period domain, Fargues-Fontaine curve, Newton stratification}

\renewcommand{\thefootnote}{\arabic{footnote}}

\begin{abstract}
	We prove the Fargues-Rapoport conjecture for $p$-adic period domains in the non-basic case with minuscule cocharacter. More precisely, we give a group theoretical criterion for the cases when the admissible locus and weakly admissible locus coincide. This generalizes the result of Hartl for the group $\mathrm{GL}_n$ (\cite{Har}). In the last section, we also give a conjecture about the intersection of weakly admissible locus and the Newton strata in the flag variety.
\end{abstract}

\maketitle
\setcounter{tocdepth}{1}
\tableofcontents

\section*{Introduction}
Let $F$ be a finite extension of $\Q_p$, and let $\breve F$ be the $p$-adic completion of the maximal unramified extension of $F$ with Frobenius $\sigma$. We consider the flag variety $\Fc(G, \mu)$ assoicated to a pair $(G, \mu)$, where $G$ is a reductive group over  $F$, and $\mu$ a minuscule cocharacter of $G$. It's a projective variety over the local reflex field $E=E(G, \{\mu\})$, the field of definition of the geometric conjugacy class $\{\mu\}$ of $\mu$. We still denote by $\Fc(G, \mu)$ the associated adic space over $\breve E$. For any $b\in G(\breve F)$ satisfying certain conditions with respective to $\mu$ (cf. \ref{section_wa}), we are interested in two open adic subspaces
\[\Fc(G, \mu, b)^{a}\subseteq \Fc(G, \mu, b)^{wa}\subseteq \Fc(G, \mu)\] inside $\Fc(G, \mu)$, where $\Fc(G, \mu, b)^{wa}$ is the weakly admissible locus, defined by Rapoport and Zink (\cite{RZ}) by removing a profinite number of Schubert varieties inside $\Fc(G, \mu)$ which contradicts the weakly admissibility condition defined by Fontaine (cf. section \ref{section_wa}), and where $\Fc(G, \mu, b)^a$ is the admissible locus, or called the $p$-adic period domain. It's much more mysterious. The existence of the admissible locus has been conjectured by Rapoport and Zink. It's characterized by the properties that it has the same classical points as the weakly admissible locus (see \cite{CoFo},  in which way weakly admissible locus is also considered as its algebraic approximation) and  there exists a local system with $G$-structures on it which interpret the crystalline representations on all classical points. When the triple $(G, \mu, b)$ is of PEL type, the admissible locus $\Fc(G,\mu,b)^a$ is the image of the $p$-adic period mapping from the Rapoport-Zink space associated to $(G, \mu, b)$ to the flag variety $\Fc(G, \mu)$. There is also direct construction of $\Fc(G, \mu, b)^a$ in special cases by Hartl \cite{Har} and Faltings \cite{Fal}.   In the most general case,  the existence of the admissible locus equipped with the \'etale local system is known due to the work of Fargues-Fontaine (\cite{FF}), Kedlaya-Liu (\cite{KL}) and Scholze (\cite{Sch1}). In fact, the admissible locus is defined by semi-stable conditions on the modification of the $G$-bundle associated to $b$ of type $\mu$ on the Fargues-Fontaine curve. Moreover, it can also be considered as the image of the $p$-adic period mapping from the local Shimura variety associated to $(G, \mu, b)$ to the flag variety $\Fc(G, \mu)$  (\cite{Sch}, \cite{Ra2}).

We want to understand the structure of the $p$-adic period domains. As the structure of its approximation, the weakly admissible locus, is well known, it's natural to ask when the $p$-adic period domain coincides with the weakly admissible locus. Hartl classified all the cases for $G=\mathrm{GL}_n$ in \cite{Har}. For general group $G$, Rapoport and Fargues have conjectured a group theoretic criterion when $b$ is basic, which has now become the following theorem.
\begin{itheorem}[Fargues-Rapoport conjecture, \cite{CFS}, Theorem 6.1]\label{theo:CFS basic}Suppose $b$ is basic. The equality
 $\Fc(G,\mu, b)^{wa}=\Fc(G,\mu, b)^{a}$ holds if and only if $(G, \mu)$ is fully Hodge-Newton decomposable.
\end{itheorem}

Recently, Shen also generalizes the Fargues-Rapoport conjecture for non-minuscule cocharacters in \cite{Sh2}. Here the fully Hodge-Newton-decomposability condition is purely group theoretic. This notion is first introduced and systematically studied by G\"ortz, He and Nie in their article \cite{GoHeNi}, where they classified all the fully Hodge-Newton decomposable pairs and give more equivalent conditions of fully Hodge-Newton decomposability. When $G$ is a general linear group, the triple $(G, \mu, b)$ is Hodge-Newton-indecomposable if all the breakpoints of the Newton polygon defined by $b$ do not touch the Hodge polygon defined by $\mu$. Otherwise, the triple is called Hodge-Newton-decomposable. Let $B(G)$ be the set of $\sigma$-conjugacy classes of $G(\breve F)$. For $b\in G(\breve F)$, let $[b]\in B(G)$ be the $\sigma$-conjugacy class of $b$. An element $[b]\in B(G)$ is called basic if the Newton polygon of $b$ is central. Kottwitz defined a subset $B(G, \mu)$ in  $B(G)$. When $G$ is the general linear group, $[b]\in B(G, \mu)$ if and only if the Newton polygon of $[b]$ lies on or above the Hodge polygon of $\mu$ and they have the same endpoint. The pair $(G, \mu)$ is fully Hodge-Newton decomposable if for any non basic $[b']$ in the Kottwitz set $B(G,\mu)$, the triple $(G, \mu, b')$ is Hodge-Newton-decomposable. We refer to section \ref{section_HN-decomposable} for the details of these notions.
\

The main result of this article is a generalized version of Fargues-Rapoport conjecture which works for any $b$. It is inspired by the Fargues-Rapoport conjecture for basic elements and Hartl's result for $\mathrm{GL}_n$. For simplicity, we assume that $G$ is quasi-split in the introduction. There is also a similar description for the non-quasi-split case (Theorem \ref{theo_main_general}) in Section \ref{section_group action on the modification}.
\begin{itheorem}[Theorem \ref{theo_main}]Suppose $M$ is the standard Levi subgroup of $G$ such that $[b]\in B(M, \mu)$ and $(M, b, \mu)$ is Hodge-Newton-indecomposable. Then the equality $\Fc(G,\mu, b)^{wa}=\Fc(G,\mu, b)^{a}$ holds if and only if  $(M,\{\mu\})$ is fully Hodge-Newton-decomposable and $[b]$ is basic in $B(M)$.
\end{itheorem}

The key ingredients of the proof of the main theorem are Proposition \ref{Prop_preimage_admissible} which describes the relation of the (weakly) admissible locus for different groups $G$ and its Levi subgroup $M$ and the following proposition.

\begin{iproposition}[Proposition. \ref{Prop_not basic}]Suppose $(G,\mu, b)$ is Hodge-Newton-indecomposable. If $b$ is not basic, then $\Fc(G, \mu, b)^a\neq \Fc(G, \mu, b)^{wa}$.
\end{iproposition}

Indeed, we know that on the flag variety there is a group action of $\tilde{J}_b$. We prove that the group action preserves the admissible locus but not the weakly admissible locus by producing a point which is weakly admissible but not admissible.  Such a point exists in the $\tilde{J}_b$-orbit of a non weakly admissible point.
\bigskip

We briefly describe the structure of this article. In section \ref{section_Kott and curve}, we review the basic facts about the Kottwitz set and $G$-bundles on the Fargues-Fontaine curve. In section \ref{section_a and wa}, we review the reduction of $G$-bundles and introduce the weakly admissible locus and admissible locus  in the flag variety in term of (weakly) semi-stable condition on the modification of $G$-bundles. In section \ref{section_HN-decomposable}, we review the Hodge-Newton-decomposability condition and prove Proposition \ref{Prop_preimage_admissible} which is one of the main ingredients for the proof of the main theorems. In section \ref{section_group action on the modification} we prove Proposition \ref{Prop_not basic} and the main theorems \ref{theo_main} and \ref{theo_main_general} by using the $\tilde{J}_b$-action on the flag variety. In section \ref{section_wa and Newton}, we discuss the relation between the Newton strata and the weakly admissible locus. We introduce a conjecture predicting which Newton strata contain weakly admissible points in the basic case. We prove this conjecture in a very special case in Proposition \ref{prop_conj special case}. The full conjecture is proved by Viehmann in \cite{Vi} very recently.

\textbf{Acknowledgments.}  We would like to thank Laurent Fargues, David Hansen, Xuhua He, Sian Nie, Xu Shen, Jilong Tong, Eva Viehmann, Bingyong Xie for many helpful discussions. Especially we would like to thank Eva Viehmann for her interest on this work and for her comments on the previous version of the preprint.  We thank the referee for careful reading and useful comments. The author is partially supported by NSFC grant No.11671136, No.12071135 and STCSM grant No.18dz2271000.

\section*{Notations}
We use the following notations:
\begin{itemize}
\item $F$ is a finite degree extension of $\Q_p$ with residue field $\F_q$ and a uniformizer $\pi_F$.
\item $\Fbar$ is an algebraic closure of $F$ and $\Gamma = \Gal (\Fbar |F)$.
\item $\breve{F} =\widehat{F^{un}}$ is the completion of the maximal unramified extension with Frobenius $\s$.
\item $G$ is a connected reductive group over $F$ and $H$ is a quasi-split inner form of $G$ equipped with an inner twisting $G_{\breve F}\xrightarrow{\sim} H_{\breve F}$.
\item $A\subseteq T \subseteq B$, where A is a maximal split torus, $T=Z_H (A)$ is the centralizer of $A$ in $T$, and $B$ is a Borel subgroup in $H$.
\item $(X^* (T),\Phi, X_*(T), \Phic )$ is the absolute root datum with positive roots $\Phi^+$ and simple roots $\Delta $ with respect to the choice of $B$.
\item $W=N_H(T)/T$ is the absolute Weyl group of $T$ in $H$, and $w_0$ is the longest length element in $W$.
\item $(X^*(A), \Phi_0,  X_*(A),\Phic_0)$ is the relative root datum with positive roots $\Phi^+_0$ and simple (reduced) roots $\Delta_0$.
\item If $M$ is a standard Levi subgroup in $H$ we denote by $\Phi_M$
 the corresponding roots or coroots showing up in $\Lie\, M$, and by $W_M$ the Weyl group of $M$. If $P$ is the standard parabolic subgroup of $H$ with Levi component $M$, sometimes we also write $W_P$ for $W_M$.
\end{itemize}

\section{Kottwitz set and $G$-bundles on the Fargues-Fontaine curve}\label{section_Kott and curve}
In this section, we will recall the some basic facts about the $G$-bundles on the Fargues-Fontaine curve which will be the main tool for our study of $p$-adic period domains.
\subsection{The Fargues-Fontaine curve}
Let $K$ be a perfectoid field over $\mathbb{F}_q$ with $\omega_K\in K$ satisfying $0<|\omega_K|<1$. Let \[\mathcal{Y}_K=\mathrm{Spa}(W_{\mathcal{O}_F}(\mathcal{O}_K))\backslash V(\pi_F[\omega_K])\] be an adic space over $F$ equipped with an automorphism $\varphi$ induced from the Frobenius $K|\mathbb{F}_q$.  The Fargues-Fontaine curve over $F$ associated to $K$ is the schema
\[X=X_{K}:=\mathrm{Proj}(\bigoplus_{d\geq 0}B_{K}^{\varphi=\pi_F^d}),\] where $B_{K}=H^0(\mathcal{Y}_K, \mathcal{O}_{\mathcal{Y}_K})$. The scheme $X$ is a curve (which means it's a one dimensional noetherian regular scheme) over $F$(\cite{FF} thm 6.5.2, 7.3.3).

If we replace $K$ by an affinoid perfectoid space $S=\mathrm{Spa}(R, R^+)$ over $\mathbb{F}_q$, we can similiarly construct $\mathcal{Y}_S$ and $X_S$ over $F$ which is called the relative Fargues-Fontaine curve (cf. \cite{KL}).

\subsection{$G$-bundles}
From now on, suppose $K=C^\flat$ is the tilt of a complete algebraically closed field $C$ over $F$. Then the curve $X$ is equipped with a closed point $\infty$ with residue field $k(\infty)=C$. Let $\mathrm{Bun}_X$ be the category of vector bundles on $X$. The classification of vector bundles on $X$ is well known due to the work of Fargues-Fontaine.

\begin{theorem}[\cite{FF} theorem 8.2.10]Every vector bundle on $X$ is a direct sum of stable sub-vector bundles, and the isomorphism classes of stable vector bundles on $X$ are parametrized by the slope in $\Q$.
\end{theorem}
For $\lambda\in\Q$, let $\mathcal{O}(\lambda)$ be a stable vector bundle on $X$ of slope $\lambda\in\Q$.
\

By \cite{FF}, $\{\infty\}=V^+(t)$ with $t\in H^0(X, \mathcal{O}(1))$. Then
\[X\backslash\{\infty\}=\Spec(B_e) \text{ and }\widehat{X_{\infty}}=\Spec(B_{dR}^+)\] where $B_e=B_K[\frac{1}{t}]^{\varphi=1}$ is a principal ideal domain, and $B_{dR}^+$ is a complete discrete valuation ring with residue field $C$. Let $B_{dR}$ be the fraction field of $B_{dR}^+$. The following proposition tells us that a vector bundle on $X$ is determined by its restrictions to $X\backslash\{\infty\}$ and $\widehat{X_{\infty}}$ with a gluing datum.

\begin{proposition}\label{prop_classification_recollement}[\cite{FF} Corollary 5.3.2]Let $\mathcal{C}$ be the category of triples $(M_e, M_{dR}, u)$ where
\begin{itemize}
\item $M_e$ is a free $B_e$-module of finite rank;
\item $M_{dR}$ is a free $B_{dR}^+$-module of finite rank;
\item $u: M_e\otimes_{B_e}B_{dR}\stackrel{\sim}{\longrightarrow}M_{dR}\otimes_{B^+_{dR}}B_{dR}$ is an isomorphism of $B_{dR}$-modules.
\end{itemize} Then there is an equivalence of categories
\[\begin{split}\mathrm{Bun}_X&\stackrel{\sim}{\longrightarrow} \mathcal{C}\\ \E&\mapsto (\Gamma(X\backslash\{\infty\},\E),  \widehat{\E}_{\infty}, \mathrm{Id}).
\end{split}\]
\end{proposition}

Let $\mathrm{Isoc}_{\breve F|F}$ be the category of isocrystals relative to $\breve{F}| F$. By \cite{FF} thm 8.2.10, there is an essentially surjective functor
\[\E(-): \mathrm{Isoc}_{\breve F|F}\rightarrow \mathrm{Bun}_X\] where for any $(D, \varphi)\in \mathrm{Isoc}_{\breve F|F}$, the vector bundle $\E(D,\varphi)$ on $X$ is associated to the graded $\bigoplus_{d\geq 0}B_{K}^{\varphi=\pi_F^d}$-module
\[\bigoplus_{d\geq 0}(D\otimes B_{K})^{\varphi\otimes\varphi=\pi_F^d}.\]In fact, it maps a simple isocrystal of slope $\lambda$ to $\mathcal{O}(-\lambda)$.

To any $b\in G(\breve F)$, we can associate an isocrystal with $G$-structures:
\[\begin{split} \mathcal{F}_b: \Rep G&\lra \mathrm{Isoc}_{\breve F|F}\\ V&\mapsto (V_{\breve F}, b\sigma).\end{split}\] Its isomorphism class only depends on the $\sigma$-conjugacy class $[b]\in B(G)$ of $b$, where $B(G)$ is the set of $\sigma$-conjugacy classes in $G(\breve F)$. In this way, $B(G)$ parametrises the set of isomorphism classes of $F$-isocrystals with $G$-structure, cf. \cite{RR} Remarks 3.4 (i).

Recall that a $G$-bundle on $X$ is a $G$-torsor on $X$ which is locally trivial for the \'etale topology. Equivalently, a $G$-bundle on $X$ can also be viewed as an exact functor $\mathrm{Rep} G\rightarrow \mathrm{Bun}_X$ where $\mathrm{Rep} G$ is the category of rational algebraic representations of $G$. The \'etale cohomology set $H^1_{\text{\'{e}t}}(X, G)$ classifies the isomorphism classes of $G$-bundles on $X$.

For $b\in G(\breve F)$, we can associate to $b$ a $G$-bundle on $X$.

\[\E_b=\E(-)\circ\mathcal{F}_b: \Rep G\stackrel{\mathcal{F}_b}{\longrightarrow} \mathrm{Isoc}_{\breve F|F}\stackrel{\E(-)}{\longrightarrow}\mathrm{Bun}_X.\] By \cite{F3} theorem 5.1, there is a bijection of sets
	\[\begin{split} B(G)&\st{\sim}{\longrightarrow} H^1_{\text{\'{e}t}}(X, G) \\
	[b]&\mapsto [\E_b]. \end{split}	\]

 In this way, the set $B(G)$ also classifies $G$-bundles on $X$.

\subsection{Kottwitz set}
There are two invariants on the set $B(G)$, the Newton map and the Kottwitz map. For any $b\in G(\breve F)$, there is a composed functor
\[\mathcal{F}:  \mathrm{Rep} G\stackrel{\mathcal{F}_b}{\longrightarrow}  \mathrm{Isoc}_{\breve F| F}\longrightarrow \mathbb{Q}-\mathrm{grVect}_{\breve F}\]
 where $\mathbb{Q}-\mathrm{grVect}_{\breve F}$ is the category of $\mathbb{Q}$-graded vector spaces over $\breve F$ and the second functor is given by the Dieudonn\'e-Manin's classification of isocrystals which decomposes an isocrystal into isocline sub-isocrystals parametrized by $\Q$.  We can attach to $\mathcal{F}$ a slope morphism
\[\nu_b: \mathbb{D}_{\breve F}=\mathrm{Aut}^{\otimes}(\omega)\rightarrow \mathrm{Aut}^{\otimes}(\omega\circ \mathcal{F})=G_{\breve F},\] where $\omega: \mathbb{Q}-\mathrm{grVect}_{\breve F}\rightarrow \mathrm{Vect}_{\breve F}$ is the natural forgetful functor and  $\mathbb{D}$ is the pro-algebraic torus over $\breve F$ with $X^*(\mathbb{D})=\mathbb{Q}$.

The conjugacy class of the slope morphism $\nu_b$ is defined over $F$ and it only depends on the $\sigma$-conjugacy class of $b$. We thus obtain the Newton map \[\begin{split} \nu: B(G)&\longrightarrow \mathcal{N}(G)\\ [b]&\mapsto [\nu_b], \end{split}\]
where $\mathcal{N}(G)=\mathcal{N} (H)=X_* (A)_\Q^+$ is
 the Newton chamber. The $\sigma$-conjugacy class $[b]\in B(G)$ is called basic if $\nu_b$ is central. Denote by $B(G)_{basic}$ the subset of basic $\sigma$-conjugacy classes in $B(G)$.

The other invariant is the Kottwitz map (\cite{Kot2} 4.9, 7.5, \cite{RR} 1.15):
\[\begin{split} \kappa_G: B(G)&\longrightarrow \pi_1(G)_{\Gamma}\\ [b]&\mapsto \kappa_G([b]),\end{split}\]
where $\pi_1(G)=\pi_1(H)=X_*(T)/\langle\Phi^\vee\rangle$ is the algebraic fundamental group of $G$, and $\pi_1(G)_{\Gamma}$ is the Galois coinvariants.  For $G=\mathrm{GL}_n$, $\kappa_{\mathrm{GL}_n}([b])=v_p(\mathrm{det}(b))\in\Z=\pi_1(\mathrm{GL}_n)_{\Gamma}$ where $v_p$ denotes the $p$-adic valuation on $\breve F$. For general $G$, the Kottwitz map is characterized by the unique natural transformation $\kappa_{(-)}: B(-)\rightarrow \pi_1(-)_{\Gamma}$ of set valued functors on the category of connected reductive groups over $F$ such that $\kappa_{\mathrm{GL}_n}$ is defined as above.
Let $B(G)_{basic}$ be the subset of basic elements in $B(G)$, then the restriction of $\kappa_G$ on $B(G)_{basic}$ induces a bijection
\[\kappa_G: B(G)_{basic}\stackrel{\sim}{\longrightarrow}\pi_1(G)_{\Gamma}.\]
Moreover if $[b]\in B(G)_{basic}$ with $\kappa_G([b])=\mu^{\sharp}\in \pi_1(G)_{\Gamma}$ with $\mu\in X_*(T)^+$, then $[\nu_b]=\mathrm{Av}_{\Gamma}\mathrm{Av}_{W}\mu$ where $\mathrm{Av}_{W}$ (resp. $\mathrm{Av}_{\Gamma}$) denotes the $W$-average (resp. $\Gamma$-average).

The elements in $B(G)$ are determined by the Newton map and Kottwitz map. Namely,  the map
\begin{eqnarray}\label{eqn_invariants Kottwitz}(\nu, \kappa_G): B(G)\longrightarrow \mathcal{N}(G)\times \pi_1(G)_{\Gamma} \end{eqnarray}is injective (\cite{Kot2} 4.13).

\begin{definition}\begin{enumerate}
\item Let $[b]\in B(G)$ and $\mu\in X_*(T)^+$, we define the Kottwitz sets
\[\begin{split}B(G, \mu)&:=\{[b]\in B(G)| [\nu_b]\leq \mu^{\diamond}, \kappa_G(b)=\mu^{\sharp}\}\\
 A(G, \mu)&:=\{[b]\in B(G)|[\nu_b]\leq \mu^{\diamond}\}\end{split}\]
which are  finite subsets in $B(G)$, where \begin{itemize}
\item $\mu^{\diamond}:=\mathrm{Av}_{\Gamma}(\mu)\in \mathcal{N}(G)$ is the Galois average of $\mu$,
\item $\mu^{\sharp}\in \pi_1(G)_{\Gamma}$ is the image of $\mu$ via the natural quotient map $X_{*}(T)\rightarrow \pi_1(G)_{\Gamma}$,
\item the order $\leq$ on $\mathcal{N}(G)$ is the usual order: $\nu_1\leq \nu_2$ if and only if $\nu_2-\nu_1 \in \Q_{\geq 0}\Phi_0^+$.
 \end{itemize}
\item We define a partial order on $A(G, \mu)$: For $[b_1], [b_2]\in A(G, \mu)$, we say $[b_1]\leq [b_2]$ if $[\nu_{b_1}]\leq [\nu_{b_2}]$.
\end{enumerate}

\end{definition}

We will also need the following generalized Kottwitz set defined in \cite{CFS}.
\begin{definition}
For $\e\in \pi_1 (G)_\Gamma$ and $\delta\in X_* (A)_\Q^+$ we set
$$
B(G,\e,\delta   ) = \{ [b]\in B(G)\ |\ \kappa_G (b)=\e \text{ and } [\nu_b]\leq \delta\}.
$$
\end{definition}
\begin{remark}\label{AGmu}\begin{enumerate}\item We following the notations in \cite{CFS}. For the generalized Kottwitz set $B(G, \e, \delta)$, the term $\e$ is written in an additive way while the term $\delta$ is written in a multiplicative way.
\item
\[\begin{split}B(G, \mu)&=B(G, \mu^{\sharp}, \mu^{\diamond})\\
A(G, \mu)&=\coprod_{\e\in\pi_1(G)_{\Gamma, tor}}B(G, \mu^{\sharp}+\e, \mu^{\diamond}),\end{split}\]
where $\pi_1(G)_{\Gamma, tor}$ denotes the subgroup of torsion elements in $\pi_1(G)_{\Gamma}$.
\end{enumerate}
\end{remark}

\begin{definition}Suppose $G$ is quasi-split. Let $M$ be a standard Levi subgroup of $G$. We define a partial order $\preceq_M$ in $\pi_1(M)_{\Gamma}$ and in $\pi_1(M)_{\Gamma,\Q}$ as follows: for $y_1, y_2\in \pi_1(M)_{\Gamma}$ (resp. $\pi_1(M)_{\Gamma,\Q}$), we write $y_1\preceq_M y_2$ if and only $y_2-y_1$ is a non-negative integral linear combination of images in $\pi_1(M)_{\Gamma}$ (resp. $\pi_1(M)_{\Gamma,\Q}$) of coroots corresponding to the simple roots of $T$ in $N$ with $N$ the unipotent radical of the standard parabolic subgroup of $G$ with Levi component $M$.
\end{definition}

We have the following characterization of the generalized Kottwitz set.

\begin{lemma}\label{lemma_BGmu}Suppose $G$ is quasi-split. Let $M$ be a standard Levi subgroup of $G$.  Let $b\in M(\breve F)$ such that $[b]_M\in B(M)$ is basic. Then
$[b]\in B(G, \e, \delta)$ if and only if \[\kappa_G (b)=\e \text{ and }\kappa_M(b)\preceq_M \delta^{\sharp} \text{ in }\pi_1(M)_{\Gamma, \Q},\] where $\delta^\sharp$ denotes for the image of $\delta$ via the natural map $X_*(A)_{\Q}\rightarrow \pi_1(M)_{\Gamma,\Q}$ by abuse of notation.

Moreover, for elements in the Kottwitz set $B(G, \mu)$, we can have an simpler characterization: $[b]\in B(G, \mu)$ if and only if $\kappa_M(b)\preceq_M \mu^{\sharp}$ in $\pi_1(M)_{\Gamma}$ where $\mu^{\sharp}$ denotes the image of $\mu$ in $\pi_1(M)_{\Gamma}$ by abuse of notation.
\end{lemma}

\begin{proof}The proof is the same as \cite{Kot3} Proposition 4.10. We repeat the proof for generalized Kottwitz set here, while the proof for $B(G, \mu)$ is similar. The necessity is obvious. For the sufficiency, the inequality $\kappa_M(b)\preceq_M \delta^{\sharp}$ in $\pi_1(M)_{\Gamma, \Q}$ implies
\[[\nu_b]=\mathrm{Av}_{\Gamma}\mathrm{Av}_{W_M}(\widetilde{\kappa_M(b)})\leq\mathrm{Av}_{\Gamma}\mathrm{Av}_{W_M}\delta\leq \delta,\]where
$\widetilde{\kappa_M(b)}$ denotes a preimage of $\kappa_M(b)$ via the natural map $X_*(T)\rightarrow \pi_1(M)_{\Gamma}$ and the second inequality is due to the fact that $\delta$ is dominant.
\end{proof}

The following lemma will be used in the proof of the main theorem.

\begin{lemma}\label{lemma_BGmu reduction Levi}Suppose $G$ is quasi-split and $\mu\in X_*(T)^+$ is minuscule. Let $M$ be a standard Levi subgroup of $G$ and $b\in M(\breve F)$.
\begin{enumerate}
\item  Then the natural map $\pi_1(M)_{\Gamma, tor}\rightarrow \pi_1(G)_{\Gamma, tor}$ is injective.
\item  Suppose $[b]\in B(G, \mu)$, then there exists $\mu'\in X_*(T)$ such that $[b]_M\in B(M, \mu')$ and $\mu'$ is conjugate to $\mu$ in $G$.
\item Suppose $[b]\in B(G, \mu^{\sharp}+\e, \mu^{\diamond})$ with $\e\in \pi_1(G)_{\Gamma, tor}$, then \[\e\in \mathrm{Im}(\pi_1(M)_{\Gamma, tor}\rightarrow \pi_1(G)_{\Gamma, tor})\] and there exists $w\in W$ such that $[b]\in B(M, (w\mu)^{\sharp, M}+\e, (w\mu)^{\diamond, M})$, here we view $\e$ as an element in $\pi_1(M)_{\Gamma, tor}$.
\end{enumerate}

\end{lemma}
\begin{proof}For (1), the map $\pi_1(M)_{\Gamma, tor}\rightarrow  \pi_1(G)_{\Gamma, tor}$ can be identified with the map $H^1(F, M)\rightarrow H^1(F, G)$ which is injective(cf. \cite{SerreCohoGal} Exercise 1 in III $\S 2.1$).

The part (2) of the lemma is proved by \cite{RV} Lemma 8.1 (2) when $G$ is unramified. For the general case, we want to reduce to the unramified case. Without loss of generality, we may assume that $G$ is adjoint and simple by \cite{Kot2} 6.5. Moreover, after replacing $M$ by a smaller Levi subgroup, we may assume that $[b]_M\in B(M)_{basic}$. Then by Lemma \ref{lemma_BGmu},
\[[b]\in B(G,\mu)\Leftrightarrow \kappa_M(b)\preceq_M \mu^{\sharp} \text{ in } \pi_1(M)_{\Gamma}.\]We want to show that there exists $\mu'\in X_*(T)$ which is conjugate to $\mu$ such that $\kappa_M(b)=(\mu')^{\sharp}$ in $\pi_1(M)_{\Gamma}$.

Let $Q:=\mathrm{Ker}(\pi_1(M)\rightarrow \pi_1(G))$.

\noindent{\it Claim: $Q_{\Gamma}=\mathrm{Ker}(\pi_1(M)_{\Gamma}\rightarrow \pi_1(G)_{\Gamma})$.}

Let $A:=\mathrm{Ker}(\pi_1(M)_{\Gamma}\rightarrow \pi_1(G)_{\Gamma})$ which is torsion free by the proof of \cite{CFS} Lemma 4.11. As the functor $(-)_{\Gamma}$ is right exact, there exists a natural surjection $Q_{\Gamma}\rightarrow A$. We need to show it's injective. Therefore it suffices to show $\mathrm{rank}_{\Q}A_{\Q}=\mathrm{rank}_{\Q}Q_{\Gamma,\Q}$. Since the functor $(-)_{\Gamma,\Q}$ is canonically isomorphic to the functor $(-)^{\Gamma}_{\Q}$, we have
\[A_{\Q}\simeq \mathrm{Ker}(\pi_1(M)^{\Gamma}_{\Q}\rightarrow \pi_1(G)^{\Gamma}_{\Q})=Q^{\Gamma}_{\Q}.\] The Claim follows.

By the Claim, $\mu-\kappa_M(b)\in Q_{\Gamma}$. We need to show there exists $\mu'\in X_*(T)$ which is conjugate to $\mu$ such that
\[ \mu-\kappa_M(b)=\mu-\mu' \text{ in } Q_{\Gamma}.\] It's a question only about root system with Galois action. Indeed, by the classification of $k$-forms of $G$, we can construct an unramified group $\tilde G$ over $F$ which is a form of $G$ and they both have same Tits-indices. More precisely, we can find $\tilde T\subset \tilde B\subset \tilde G$  over $F$ where $\tilde{T}$ is a maximal torus and $\tilde B$ is a Borel subgroup such that
\begin{itemize}
\item $X^*(T)\simeq X^*(\tilde T)$ and via this identification $\Delta_G=\Delta_{\tilde G}$;
\item $\Delta_{G}$ and $\Delta_{\tilde G}$ have the same Galois orbits.
\end{itemize}
Then the absolute Weyl group of $(G, T)$ and $(\tilde G, \tilde T)$ are isomorphic. Let $\tilde M$ be the standard Levi subgroup of $\tilde G$ such that $\Delta_{\tilde M}=\Delta_M$. The isomorphism between the character groups induces an identification $\pi_1(\tilde M)_{\Gamma}=\pi_1(M)_{\Gamma}$. Let $\tilde\mu\in X_*(\tilde T)$ be the cocharacter corresponding to $\mu$ via the idendtification $X_*(T)\simeq X_*(\tilde T)$.  Let $[\tilde{b}]\in B(\tilde M)_{basic}$ such that $\kappa_{\tilde M}(\tilde b)\in \pi_1(\tilde M)_{\Gamma}$ maps to $\kappa_M(b)\in \pi_1(M)_{\Gamma}$ via the identication $\pi_1(\tilde M)_{\Gamma}=\pi_1(M)_{\Gamma}$. As $\tilde G$ is unramified, we can find $\tilde \mu'$ for $[\tilde b]\in B(\tilde G, \tilde \mu)$. Then $\mu'$ is the cocharcter of $G$ corresponding to $\tilde\mu'$.

For (3), as before, we may assume $[b]\in B(M)_{basic}$ after replacing $M$ by a smaller group. Then $[\nu_b]\leq \mu^{\diamond}$ implies that
\[\kappa_M(b)\preceq_M \mu^{\sharp, M} \text{ in } \pi_1(M)_{\Gamma, \Q}.\]Hence there exists $\e'\in \pi_1(M)_{\Gamma, tor}$ such that
\[\kappa_M(b)-\e'\preceq_M \mu^{\sharp, M}\text{ in } \pi_1(M)_{\Gamma}.\] Let $[b']\in B(M)_{basic}$ such that $\kappa_M(b')=\kappa_M(b)-\e'$. Then $[b']\in B(G, \mu)$. By (2), there exists $w\in W$ such that $[b']\in B(M, w\mu)$. Therefore  $[\nu_b]_M=[\nu_{b'}]_M\leq (w\mu)^{\diamond, M}$ and \[\mu^{\sharp}=\kappa_G(b')=\kappa_{G}(b)-\e'=\mu^{\sharp}+\e-\e' \text{ in } \pi_1(G)_{\Gamma},\]where $\e'$ is considered to be an element in $\pi_1(G)_{\Gamma, tor}$ via the natural map in (1). Hence $\e=\e'$ and
\[\kappa_M(b)=\kappa_M(b')+\e =(w\mu)^{\sharp}+\e \text{ in } \pi_1(M)_{\Gamma}.\]
\end{proof}


\subsection{Classification of $G$-bundles in terms of $\varphi$-modules over $\bar{B}$}
  Let \[\begin{split}B^{b, +}&:= W_{\mathcal{O}_F}(\mathcal{O}_K)[\frac{1}{\pi_F}],\\ \bar{B}&:=(B^{b, +}/[\omega_K])_{red}. \end{split}\] Here  $\bar B$ is a local $F$-algebra with residue field $W_{\mathcal{O}_F}(k_K)_{\Q}$.

 The Frobenius on $\mathcal{O}_K$ induces an automorphism $\varphi$ on $B^{b, +}$ and on $\bar B$.

Let $\varphi-\Mod_{\bar B}$ (resp. $\varphi-\Mod_{W_{\mathcal{O}_F}(k_K)_{\Q}}$) be the category of free $\bar{B}$-modules (resp. $W_{\mathcal{O}_F}(k_K)_{\Q}$-vector spaces) of finite rank equipped with a semi-linear isomorphism.

\begin{theorem}[\cite{FF} Theorem 11.1.7 and 11.1.9]\label{theorem_classification_barB-modules} There is an equivalence of additive tensor categories:
\[\mathrm{Bun}_X\stackrel{\sim}{\longrightarrow} \varphi-\Mod_{\bar B}.\]
\end{theorem}

 For $(M, \varphi)\in \varphi-\Mod_{\bar B}$, the Harder-Narasimham filtration of the corresponding vector bundle gives a $\Q$-filtration $(M^{\geq \lambda})_{\lambda\in\Q}$ of $M$ which is called the Harder-Narasimham filtration of $M$ (cf. \cite{F3} 5.4.1).

For any $\beta\in G(\bar B)$, we define
\[\begin{split}\E_{\beta}: \Rep G&\longrightarrow \varphi-\Mod_{\bar B}\stackrel{\sim}{\longrightarrow}\Bun_X\\
(V, \rho)&\mapsto (V\otimes_F \bar B, \rho(\beta)\varphi)\end{split}\]

\begin{proposition}[\cite{F3} Proposition 5.11]\label{prop_classification_Bbar} The functor $\beta\mapsto \E_{\beta}$ induces a bijection between the set of $\varphi$-conjugacy classes in $G(\bar B)$ and the set of isomorphism classes of $G$-bundles on $X$.
\end{proposition}

We also define a functor $\mathrm{red}_{\bar B, \breve F}$ as composition of two functors:
\[\mathrm{red}_{\bar B, \breve F}: \varphi-\Mod_{\bar B}\stackrel{\otimes_{\bar B} W_{\mathcal{O}_F}(k_K)_{\Q}}{\longrightarrow}\varphi-\Mod_{W_{\mathcal{O}_F}(k_K)_{\Q}}\stackrel{\sim}{\longrightarrow} \mathrm{Isoc}_{\breve F|F}\]
where the second functor is a quasi-inverse of the functor
\[(-)\otimes_{\breve F} W_{\mathcal{O}_F}(k_K)_{\Q}:  \mathrm{Isoc}_{\breve F|F}\stackrel{\sim}{\longrightarrow}\varphi-\Mod_{W_{\mathcal{O}_F}(k_K)_{\Q}}\] which is an equivalence of categories due to Dieudonn\'e-Manin's theorem of classification of isocrystals.

\subsection{The automorphism group $\tilde{J}_b$}\label{subsection_Jb}
For $[b]\in B(G)$, let  $\tilde{J}_{b}=\underline{\mathrm{Aut}} (\E_{b})$ be the pro-\'etale sheaf of automorphisms of $\E_{b}$ on the category of affinoid perfectoid spaces $\mathrm{Perf}_{\overline{\F}_q}$ over $\overline{\F}_q$. More precisely, for any affinoid perfectoid space $S$ over $\overline{\F}_q$,  one has $\tilde{J}_b(S)=\mathrm{Aut}(\E_{b|X_S})$.


 In this subsection, we review the structure of the group $\tilde{J}_b(K)$ studied in \cite{F3} section 5.4.2. Suppose $\E_{b}$ corresponds to the $\varphi$-conjugay class of $\beta\in G(\bar B)$ as in Proposition \ref{prop_classification_Bbar}. Then
\[\tilde{J}_b(K)\simeq \{g\in G(\bar B)| g\beta=\beta\varphi(g)\}.\] We will identify these two groups via this isomorphism. In order to study the structure of $\tilde{J}_b(K)$, we need to use a parabolic subgroup of $G\otimes \bar B$ that contains $\tilde{J}_b(K)$.

Consider the functor
\[\begin{split}\mathrm{Rep}G&\longrightarrow\varphi-\Mod_{\bar B}\longrightarrow \Q-\text{filtered } \bar B-\text{modules}\\
(V, \rho)&\mapsto (V\otimes_F \bar B, \rho(\beta)\varphi) \end{split}\] where the second functor is given by the Harder-Narasimham filtration. By \cite{Z} theorem 4.40, this functor corresponds to a parabolic subgroup $\mathcal{P}\subset G\otimes_F \bar B$ with $\tilde{J}_b(K)\subset\mathcal{P}(\bar B)$. The structure of $\mathcal{P}$ is well understood. Let \[\mathrm{Ad}: G\rightarrow \mathrm{GL}(\mathfrak{g})\] be the adjoint representation with $\mathfrak{g}:=\Lie G$. Then $(\mathfrak{g}\otimes_F\bar B, \mathrm{Ad}(\beta)\varphi)$ is the $\varphi$-module over $\bar B$ corresponds to the vector bundle $\mathrm{Ad}(\E_{\beta}):=\E_{\beta}\times^{G, \mathrm{Ad}}\mathfrak{g}$. Hence it has the Harder-Narasimham filtration $(\mathfrak{g}^{\geq\lambda}_{\bar B})_{\lambda\in\Q}$. In particular, for $\lambda\neq 0$, the dimension of $\mathrm{gr}^\lambda \mathfrak{g}_{\bar B}$ equals to the number of roots $\alpha\in\Phi$ such that $\langle\alpha, \nu_b\rangle=\lambda$.  Then
\[\begin{split}\mathcal{P}&=\{g\in G_{\bar B}| \mathrm{Ad}(g)(\mathfrak{g}_{\bar B}^{\geq \bullet})=\mathfrak{g}_{\bar B}^{\geq \bullet}\}\\
\mathrm{Lie}\mathcal{P}&=\mathfrak{g}_{\bar B}^{\geq 0}\end{split}\] Moreover,  the parabolic subgroup $\mathcal{P}$ is filtered by $(\mathcal{P}^{\geq \lambda})_{\lambda\in\Q_{\geq 0}}$ such that
\[\begin{split}\mathcal{P}^{>0}&=R_u\mathcal{P};\\
\forall \lambda>0, \mathcal{P}^{\geq \lambda}/\mathcal{P}^{>\lambda}&\stackrel{\sim}{\rightarrow}\mathrm{gr}^{\lambda}\mathfrak{g}_{\bar B}\otimes \mathbb{G}_a;\\
\mathcal{P}^{\geq\lambda}&=\{g\in G_{\bar B}| (\mathrm{Ad}(g)-\mathrm{Id})(\mathfrak{g}_{\bar B}^{\geq \bullet})=\mathfrak{g}_{\bar B}^{\geq \bullet+\lambda}\}.
\end{split}\]
Let $\tilde{J}_b^{\geq \lambda}(K)=\tilde{J}_b(K)\cap \mathcal{P}^{\geq \lambda}(\bar B)$ for all $\lambda\in\Q_{\geq 0}$, then we can understand the graded pieces:
\[\begin{split}\tilde{J}_b(K)/\tilde{J}_b^{>0}(K)&\simeq J_b=\{g\in G(\breve{F})| b\sigma(g)=gb\};\\
\forall \lambda>0, \tilde{J}_b^{\geq \lambda}(K)/\tilde{J}_b^{>\lambda}(K)&\simeq (\mathrm{gr}^{\lambda}\mathfrak{g}_{\bar B})^{\mathrm{Ad}(\beta)\varphi=\mathrm{Id}}\end{split}\]
where $(\mathrm{gr}^{\lambda}\mathfrak{g}_{\bar B})^{\mathrm{Ad}(\beta)\varphi=\mathrm{Id}}$ is $\frac{\dim \mathrm{gr}^{\lambda}\mathfrak{g}_{\bar B} }{h}$ copies of $H^0(X, \mathcal{O}_X(\lambda))$ if $\lambda=\frac{d}{h}$ with $(d, h)=1$. In particular $\tilde{J}_b^{\geq \lambda}(K)\varsupsetneqq \tilde{J}_b^{>\lambda}(K)$ if there exists $\alpha\in\Phi$ such that $\langle\alpha, \nu_b\rangle=\lambda>0$.
\subsection{Modifications of a $G$-bundle on $X$}

\begin{definition}Let $\E$ be a $G$-bundle on $X$. A modification of $G$-bundles of $\E$ (on $\infty$) is a pair $(\E', u)$, where $\E'$
is a $G$-bundles on $X$ and \[u:\E|_{X\backslash\{\infty\}}\stackrel{\sim}{\longrightarrow}\E'|_{X\backslash\{\infty\}} \] is an isomorphisme
of $G$-bundles on $X\backslash\{\infty\}$. Two modifications $(\E', u)$ and $(\tilde{\E}', \tilde{u})$ of $\E$ are said to be equivalent if there exists an isomorphism $f: \E'\stackrel{\sim}{\rightarrow}\tilde{\E}'$ such that $\tilde{u}=f|_{X\backslash\{\infty\}}\circ u$.
\end{definition}

Consider the $B_{dR}$-affine Grassmannian $\Gr_G^{B_{dR}}$ attached to $G$ (cf. \cite{Sch}). We only need its $C$-points
\[\mathrm{Gr}_G^{dR}(C):=G(B_{dR})/G(B_{dR}^+).\]

For any $b\in B(G)$, let $\E_b$ be the associated $G$-bundle on $X$. For any $x\in\mathrm{Gr}_G^{dR}(C)$, we can construct a modification $\E_{b,x}$ of $\E_b$ \`a la Beauville-Laszlo given by gluing $\E_{b|X\backslash\{\infty\}}$ and the trivial bundle on $\Spec(B_{dR}^+)$ via the gluing datum given by $x$ (cf. \cite{CS} Theorem 3.4.5 and \cite{F2} 4.2, \cite{F} Proposition 3.20). Moreover, by \cite{F} Proposition 3.20, there is a bijection
\begin{eqnarray}\label{dR_modification}\mathrm{Gr}_G^{dR}(C)&\stackrel{\sim}{\longrightarrow}&\{\mathrm{equivalent\ classes\ of\ modifications\ of\ } \E_b\}\\
\nonumber x&\longmapsto& \mathrm{equivalent\ class\ of\ } (\E_{b, x}, \mathrm{Id}) \end{eqnarray}

For $\mu\in X_*(T)^+$, the corresponding affine Schubert cell is
\[\mathrm{Gr}_{G, \mu}^{B_{dR}}(C)=G(B_{dR}^+)\mu(t)^{-1}G(B_{dR}^+)/G(B_{dR}^+)\subset \mathrm{Gr}_G^{B_{dR}}(C).\]

Here we use the non standard notion of affine Schubert cell associated to anti-dominant $\mu^{-1}$. This affine Schubert cell is closely related to the modification of $G$-bundles of type $\mu$ as in the following definition.

\begin{definition}A modification of $\E_b$ is of type $\mu$ if its equivalent class falls in the affine Schubert cell $\mathrm{Gr}_{G, \mu}^{B_{dR}}(C)$ via (\ref{dR_modification}).
\end{definition}

The natural action of $\tilde{J}_b(K)=\mathrm{Aut}(\E_b)$ on the set of modifications of $\E_b$ induces via (\ref{dR_modification}) an action of $\tilde{J}_b(K)$ on $\mathrm{Gr}_G^{B_{dR}}(C)$.

Let $\widehat{\E_{b}}_\infty$ be the local completion of $\E_b$ at $\infty$. It is canonically trivialized. Hence there is a natural morphism
\begin{eqnarray*}\label{morphism_tilde J_b}\alpha_{b, G}: \tilde J_b(K) =\mathrm{Aut}(\E_b)\rightarrow \mathrm{Aut}(\widehat{\E_{b}}_\infty)=G(\BpdR) \end{eqnarray*}

The action of $\tilde{J}_b(K)$ on $\mathrm{Gr}_G^{B_{dR}}(C)$ is given by left multiplication via $\alpha_{b, G}$.

\begin{lemma}\label{lemma_Jbaction} Let $\gamma\in \tilde{J}_b(K)$.  For any $x\in \mathrm{Gr}_G^{B_{dR}}(C)$, the automorphism $\gamma: \E_b\stackrel{\sim}{\rightarrow}\E_b$ induces an automorphism  \[\tilde{\gamma}: \E_{b, x}\stackrel{\sim}{\rightarrow} \E_{b, \gamma(x)}\] such that the following diagram is commutative:
\[\xymatrix{\E_{b|X\backslash\{\infty\}}\ar[r]^{\gamma}\ar[d]_{\mathrm{Id}}&\E_{b|X\backslash\{\infty\}}\ar[d]^{\mathrm{Id}}\\ \E_{b,x|X\backslash\{\infty\}}\ar[r]^{\tilde{\gamma}}&\E_{b, \gamma(x)| X\backslash\{\infty\}}}.\]\end{lemma}
\begin{proof}By Tannakian formalism, it suffices to deal with the case when $G=\mathrm{GL}_n$. Then $G$-bundles on $X$ is the same thing as vector bundles of rank $n$. Suppose $\E_b$ corresponds to the triple $(M_e, M_{dR}, u)$ as in Proposition \ref{prop_classification_recollement}, then $x\in \mathrm{Gr}_G^{B_{dR}}(C)$ corresponds to a $B_{dR}^+$-lattice $M_x$ in $M_{dR}\otimes_{B^+_{dR}}B_{dR}$ and $\E_{b, x}$ corresponds to the triple $(M_e, M_x, u)$. The automorphism $\gamma: \E_b\stackrel{\sim}{\rightarrow}\E_b$ corresponds to a pair $(\gamma_e, \gamma_{dR})$ of automorphisms compatible with $u$, where \[\gamma_e: M_e\stackrel{\sim}{\longrightarrow}M_e, \gamma_{dR}: M_{dR}\stackrel{\sim}{\longrightarrow}M_{dR}.\] Then $\E_{b, \gamma(x)}$ corresponds to the triple $(M_e, \gamma_{dR}(M_x), u)$ where $\gamma_{dR}(M_x)$ is the image of $M_x$ via $\gamma_{dR}\otimes_{B_{dR}^+}B_{dR}: M_{dR}\otimes_{B_{dR}^+}B_{dR}\stackrel{\sim}{\rightarrow} M_{dR}\otimes_{B_{dR}^+}B_{dR}$. We define $\tilde{\gamma}: \E_{b, x}\stackrel{\sim}{\rightarrow} \E_{b, \gamma(x)} $ to be the automorphism corresponds to \[\gamma_e: M_e\stackrel{\sim}{\longrightarrow}M_e, \gamma_{dR}: M_x\stackrel{\sim}{\longrightarrow}\gamma_{dR}(M_x).\]The commutativity of the diagram can be verified directly.
\end{proof}

Recall that we have the {\it Bialynicki-Birula map } (cf. \cite{CS} Propostion 3.4.3)
$$
\pi_{G,\mu}:\Gr^{\mathrm{B_{dR}}}_{G,\mu} (C) \longrightarrow \Fc  (G,\mu) (C).
$$
By Tannakian formalism, we may reduce the construction to the case when $G=\mathrm{GL}_n$. In this case, $\Gr^{\mathrm{B_{dR}}}_{G,\mu} (C)$ paramatrizes the lattices in $\mathrm{B}_{dR}^n$ that has relative position ${\mu}^{-1}$ with the standard lattice $\mathrm{B}_{dR}^{+, n}$. Write $\mu=(k_1, \cdots, k_n)$ with $k_1\geq k_2\geq\cdots\geq k_n$. Suppose $\Lambda\in \Gr^{\mathrm{B_{dR}}}_{G,\mu} (C)$. We define an increasing filtration $\mathrm{Fil}_{\Lambda}^{\bullet}$ of $C^n$ as follows: for any $m\in\Z$,
\[\mathrm{Fil}_{\Lambda}^m C^n:= (\mathrm{B}_{dR}^+)^n\cap t^{-m}\Lambda/((t\mathrm{B}_{dR}^+)^n\cap t^{-m}\Lambda)\subseteq (\mathrm{B}_{dR}^+)^n/(t\mathrm{B}_{dR}^+)^n=C^n.\] It's easy to check that \[\mathrm{dim}_C \mathrm{Fil}_{\Lambda}^m C^n=\mathrm{max}\{1\leq i\leq n| k_i\geq -m\}.\] Therefore $\pi_{G,\mu}(\Lambda):=\mathrm{Fil}_{\Lambda}^{\bullet}\in \Fc(G, \mu)(C)$.

From now on, suppose $\mu$ is minuscule. Then the Bialynicki-Birula map $\pi_{G, \mu}$
is an isomorphism by \cite{CS} Lemma 3.4.4. For $x\in\Fc(G, \mu)(C)$, we denote by $\E_{b, x}$ the modification $\E_{b, \pi_{G, \mu}^{-1}(x)}$ of $\E_b$ of type $\mu$.

When $[b]\in B(G)$ is basic, the isomorphism classes of the modifications of $\E_b$ can be classified as follows.

 \begin{proposition}[\cite{Ra2} A.10, \cite{CFS} Prop. 5.2]\label{prop:description modification basique}
 Let $[b]\in B(G)$ be basic. Let \[B(G, \kappa_G(b)- \mu^\sharp, \nu_b\mu^{-1}):=B(G, \kappa_G(b)-\mu^\sharp, \nu_b(w_0\mu^{-1})^{\diamond})\subseteq B(G).\](Here we write the element $\nu_b(w_0\mu^{-1})^{\diamond}$ in $X_*(A)^+_{\Q}$ in the multiplicative form but not the usual additive form.) The map $[b']\mapsto [\E_{b'}]$ gives a bijection
\[B(G, \kappa_G(b)- \mu^\sharp, \nu_b\mu^{-1})\simeq \{ \E_{b,x}|\; x\in  \Fc(G,\mu)(C)\} /\sim . \]
\end{proposition}


The action of $\tilde J_b (K)$ on  $\Gr^{\mathrm{B_{dR}}}_{G,\mu} (C)$ defined by the multiplication on the left via the morphism $\alpha_{b, G}$ induces an action of $\tilde J_b (K)$ on $\Fc(G, \mu)(C)$ via the Bialynicki-Birula map $\pi_{G,\mu}$. For any $\gamma\in\tilde J_b(K)$, we have an automorphism (still denoted by) \[\tilde{\gamma}: \E_{b, x}\stackrel{\sim}{\rightarrow}\E_{b, \gamma(x)} \] of $G$-bundles for any $x\in \Fc(G, \mu)(C)$.


\section{Admissible locus and weakly admissible locus}\label{section_a and wa}
\subsection{Reductions of $G$-bundles}
\begin{definition}\begin{enumerate}
                    \item Let $H\subseteq G $ be a closed subgroup of $G$. Suppose $\E$ is a $G$-bundle on $X$. A reduction of $\E$ to $H$ is a pair $(\E_H, \iota)$ where $\E_H$ is a $H$-bundle and $\iota: \E_H\times^H G\stackrel{\sim}{\rightarrow}\E$ is an isomorphism of $G$-bundles. We will also write $\E_H$ for such a reduction if we don't need to emphasis $\iota$.
\item Two reductions $(\E_H, \iota)$ and $(\E_H', \iota')$ of $\E$ to $H$ are called equivalent if there exists an isomorphism $u: \E_H\stackrel{\sim}{\rightarrow}\E_H'$ such that $\iota=\iota'\circ (u\times^H G)$.
\end{enumerate}
\end{definition}

\begin{remark}\label{remark_section}The equivalence classes of redcutions of $\E$ to $H$ are in bijection with the sections of the fibration $H\backslash\E\rightarrow X$.
\end{remark}

We will assume $G$ quasi-split in the rest of this subsection.

\begin{definition}Let $b\in G(\breve F)$. For a Levi subgroup $M$ of $G$, a reduction of $b$ to $M$ is a pair $(b_M, g)$ with $b_M\in M(\breve F)$ and $g\in G(\breve F)$ such that $b=gb_M\sigma(g)^{-1}$. We also write $b_M$ for such a reduction if we don't need to emphasize $g$. Two reductions $(b_M, g)$ and $(b_M', g')$ of $b$ to $M$ are equivalent if and only if there exists $h\in M(\breve F)$ such that $(b_M', g')=(hb_M\sigma(h)^{-1}, gh^{-1})$. Similarly, we can define the same notion for parabolic subgroups.  \end{definition}

There is a natural injective map
\[\{\text{equivalent classes of reductions of } b \text{ to } M \}\longrightarrow \{\text{equivalent classes of reductions of } \E_b \text{ to } M \}.\] This map is in general not surjective.

\begin{example}Let $G=\mathrm{GL}_5$ with Levi subgroup $M=\mathrm{GL}_3\times \mathrm{GL}_2$. Let $b\in G(\breve F)$ with Newton slopes $(\frac{1}{3}, \frac{1}{3}, \frac{1}{3}, \frac{1}{2}, \frac{1}{2})$. Then there exists a unique equivalent class of reductions of $b$ to $M$. However, as the decomposition of $\E_b=\mathcal{O}(-\frac{1}{3})\oplus \mathcal{O}(-\frac{1}{2})$ on semi-stable vector bundles is not canonical because of the existence of morphisms $\mathcal{O}(-\frac{1}{2})\rightarrow \mathcal{O}(-\frac{1}{3})$, there exist infinite equivalent classes of reductions of $\E_b$ to $M$.
\end{example}

The following lemma will be used frequently in the sequel.
\begin{lemma}
\label{lemma_redction modfication}[\cite{CFS} Lemma 2.4]
Let $\E$ and $\E'$ be two $G$-bundles on $X$ with a modification $\E|_{X\setminus\{\infty\}}\st{\sim}{\ra}\E'|_{X\setminus\{\infty\}}$. Then for any parabolic subgroup $P$ of $G$, we have a bijection
\[ \{\tr{Reductions of }\; \E \tr{ to }\; P\}\lra \{ \tr{Reductions of }\; \E'\;\tr{ to }\; P\}.\]
\end{lemma}

Let $\E$ be a $G$-bundle on $X$, by \cite{F3} 5.1, there exists the canonical reduction $\E_P$ of $\E$ to a unique standard parabolic subgroup $P$ of $G$ such that
\begin{itemize}
\item the assoicated $M$-bundle $\E_P\times^{P}M$ is semi-stable,  where $M$ is the Levi component of $P$,
\item  for any $\chi\in X^*(P/Z_G)\backslash\{0\}\cap \mathbb{N}\Delta_G$,  we have $\deg \chi_*\E_P>0$.
\end{itemize}
Using the Harder-Narasimham reduction $\E_P$, we can define the slope
\[\nu_{\E}\in X_*(A)_{\Q}\] of $\E$ by the Galois invariant morphism $X^*(P)\rightarrow \Z$ which maps $\chi\in X^*(P)$ to $\deg\chi_*\E_P$ combined with the inclusion \[\Hom_{\Z}(X^*(P), \Z)^{\Gamma}=X_*(M^{ab})^{\Gamma}\subset X_*(A_M)_{\Q}\subseteq X_*(A)_{\Q}, \]where $M^{ab}$ is the cocenter of $M$ and $A_M\subseteq A$ is a maximal split central torus of $M$.

\begin{proposition}[\cite{FF} Proposition 6.6, \cite{CS} Lemma 3.5.5]\label{prop_invariant_bundle}Let $[b]\in B(G)$ and $x\in \mathrm{Gr}_{G, \mu}^{B_{dR}}(C)$.  Then
\begin{enumerate}
\item we have an equality in the positive Weyl chamber
\[\nu_{\E_b}=-w_0[\nu_b],\] where $w_0$ is the element of longest length in the Weyl group $W$;
\item \[c_1^G(\E_{b,x})=\mu^{\sharp}-\kappa_G(b)\in \pi_1(G)_{\Gamma},\]where $c_1^G$ denotes the $G$-equivariant first Chern class. In particular,  $c_1^G(\E_b)=-\kappa_G(b)$.
\end{enumerate}
\end{proposition}

Recall the following fact:
\begin{theorem}[\cite{Schi} theo. 4.5.1]
\label{theoSchi}
Let $\E$ be a $G$-bundle on $X$.
\begin{enumerate}
\item Suppose $\E_Q$ is a reduction of $\E$ to the standard parabolic subgroup $Q$. Consider the vector
\begin{align*}
v: X^*(Q) & \longrightarrow \Z \\
\chi & \longmapsto \deg \, \chi_* \E_Q
\end{align*}
seen as an element of $X_*(A)_\Q$.  Then one has $v\leq \nu_\E$. Moreover, if this inequality is an equality,  then $Q\subset P$ and $\E_P \simeq \E_Q\times^Q P$, where $\E_P$ is the canonical reduction of $\E$.
\item The vector $\nu_\E$ can be defined as being the supremum of all such vectors $v$ associated to all possible reductions $\E_Q$ in the poset $X_* (A)_\Q$.

\end{enumerate}
\end{theorem}

\begin{remark}\label{remark_vector v}If we view $v$ as an element in $X_*(A)_{\Q}$, then $v=\mathrm{Av}_{W_{M_Q}}(\nu_{\E_{M_Q}})$, where $M_Q$ is the Levi component of $Q$ and $\E_{M_Q}=\E_{Q}\times^{Q}M$.
 \end{remark}

\begin{corollary}\label{Coro_HN equality among reduction to M}Suppose $\E$ is a $G$-bundle with $\E_{P}$ a reduction to $P$. Let $\E'=(\E_{P}\times^P M)\times^M G$. Then $\nu_{\E}\preceq \nu_{\E'}$. In particular, if $\E_{P}\times^P M$ is a trivial $M$-bundle, then $\E$ is a trivial $G$-bundle.
\end{corollary}

\begin{remark}When $G=\mathrm{GL}_n$, this corollary is shown in \cite{Ked} Lemma 3.4.17.
\end{remark}

\begin{proof}[Proof of Corollary \ref{Coro_HN equality among reduction to M}]Suppose $\E_Q$ is a reduction of $\E$ to a standard parabolic subgroup $Q$.
Suppose \[s_P: X\rightarrow P\backslash\E,\  s_Q: X\rightarrow Q\backslash\E\] are the corresponding sections for $\E_P$ and $\E_Q$ respectively. Then the relative position map:
\begin{eqnarray*}Q\backslash G\times P\backslash G &\rightarrow& Q\backslash G/P= W_Q\backslash W/W_P\\
 (Q g_1, Pg_2)&\mapsto& Qg_1g_2^{-1}P\end{eqnarray*}
gives
\[X\stackrel{(s_Q, s_P)}{\longrightarrow}Q\backslash\E\times P\backslash\E\rightarrow W_Q\backslash W/W_P.\]
Let $W_Q\dot wW_P$ be the image of the generic point of $X$, where $\dot w$ is the minimal length representative of the double coset $W_Q\dot wW_P$. Let $X'\subseteq X$ be the preimage of $W_Q\dot w W_P$. It's an open subscheme of $X$.

 By Theorem~\ref{theoSchi}(2), the corollary follows from the following Claim.

\textit{Claim: $\E_Q$ induces a reduction $\E'_Q$ of $\E'$ to the parabolic subgroup $Q$ such that $\E_Q$ and $\E'_Q$ has the same vector $v$ as defined in Theorem~\ref{theoSchi}.}

Now it remains to prove the Claim. The composition of morphisms \[\E_P\rightarrow\E\stackrel{\times \dot w}{\rightarrow}{\E}\] induces a monomorphism
\[P\cap \dot w^{-1}Q\dot w\backslash \E_P\rightarrow Q\backslash\E.\] The pullback by the section $s_Q: X\rightarrow Q\backslash\E$ of this morphism gives a section of \[P\cap \dot w^{-1}Q\dot w\backslash (\E_{P})_{|X'}\rightarrow X'.\]

Combined with the natural morphism \[P\cap \dot w^{-1}Q\dot w\backslash \E_P\rightarrow P_M'\backslash\E'_{M}\] induced by the projection to the Levi quotient,  where $\E'_M=\E_P\times^P M$ and $P'_M=M\cap \dot w^{-1}Q\dot w$, we get a reduction $(\E'_{M, P_M'})_{|X'}$ of $(\E'_{M})_{|X'}$ to its standard parabolic subgroup $P'_M$. The composition of morphisms \[\E'_M\rightarrow\E'\stackrel{\times \dot w}{\rightarrow}{\E'}\] induces a morphism
\[(P'_M\backslash \E'_{M})_{|X'}\rightarrow (Q\backslash\E')_{|X'},\] and we get a reduction $\E'_{|X', Q}$ of $\E'_{|X'}$ to $Q$ induced from $\E'_{M, P_M'}$ via this morphism. Then the desired reduction $\E'_Q$ of the Claim is obtained by the application of
the valuative criterion of properness to $Q\backslash \E'\rightarrow X$ as $Q\backslash G$ is proper and $X$ is a Dedekind scheme.

Now it remains to show that $\E_Q$ and $\E'_Q$ has the same vector $v$ as defined in Theorem~\ref{theoSchi}. By the construction of $\E'_Q$, this results from the commutativity of the following diagram for any $\chi\in X^*(Q)$
\begin{eqnarray}\label{diagram in corollary}\xymatrix{P\cap \dot w^{-1}Q\dot w\ar[r]^-{\mathrm{ad} \dot w}\ar[d]_{\mathrm{pr}_M} &Q\ar[r]^{\chi} &\mathbb{G}_m\ar@{=}[d]\\ P'_M\ar[r]^{\mathrm{ad}\dot w} &Q\ar[r]^{\chi} &\mathbb{G}_m}
\end{eqnarray} where $\mathrm{pr}_M$ is restriction to $P\cap w^{-1}Qw$ of the projection of $P$ to its Levi component $M$.
\end{proof}

\begin{remark}Notations as in the above corollary. When $G=\mathrm{GL}_n$ and $M=\mathrm{GL}_{n_1}\times \mathrm{GL}_{n_2}$ with $n_1+n_2=n$, then $\E$ corresponds to a vector bundle of rank $n$ over $X$, and $\E\times^P M$ corresponds to a pair of a vector bundle of rank $n_1$ and a vector bundle of rank $n_2$ over $X$ which admits an extension by $\E$. Therefore the above corollary gives a necessary condition of whether a vector bundle over Fargues-Fontaine curve $X$ is an extension of two given vector bundles over $X$.
\end{remark}

The converse side of the corollary is the following conjecture.

\begin{conjecture}Let $P$ be a standard parabolic subgroup of $G$ with Levi component $M$. Let $\E$ be a $G$-bundle and $\E'_M$ be a semi-stable $M$-bundle with $\nu_{\E}\preceq \nu_{\E'}$ where $\E'=\E'_M\times^M G$. Suppose $\nu_{\E'_M}$ is $G$-anti-dominant. Then $\E$ has a reduction $\E_P$ to $P$ such that $\E_P\times^P M\simeq \E_M'$.
\end{conjecture}

\begin{remark}When $G=\mathrm{GL}_n$, the conjecture is proved in \cite{BFH}. Recently, this conjecture is proved by Viehmann in \cite{Vi}. Note that in the conjecture the $G$-anti-dominant assumption for $\nu_{\E'_M}$ is necessary since $H^1(X, \mathcal{O}(\lambda))=0$ for $\lambda\geq 0$. The assumption that $\E'_M$ is semi-stable (i.e. which corresponds to a basic element in $B(M)$) is also necessary which will be explained in a joint work in preparation by Jilong Tong and the author.
\end{remark}

Let $b_M$ be a reduction of $b$ to $M$, where $M$ is a standard Levi subgroup of $G$. Let $P$ be the standard parabolic subgroup with Levi component $M$.
Recall that for any $w\in W$,  there is an affine fibration
\[\mathrm{pr}_{w}:\Fc(G, \mu)(C)^{w}:=P(C)wP_{\mu}(C)/P_{\mu}(C)\ra \Fc(M,w\mu)(C)\] by projection to the Levi quotient.  We have the following fact.
\begin{lemma}[\cite{CFS} Lemma 2.6]\label{lemma:reduction to P modification}
For $x\in P(C)w P_{\mu}(C)/P_{\mu}(C)$ there is an isomorphism
$$
(\E_{b,x})_P \times^P M \simeq \E_{b_M,\mathrm{pr}_w (x)}.
$$ where $b_M$ is a reduction of $b$ to $M$ and $(\E_{b,x})_P$ is the reduction of $\E_{b, x}$ to $P$ induced by the reduction $\E_{b_M}\times^M P$ of $\E_b$ to $P$.
\end{lemma}

\subsection{Weakly admissible locus}\label{section_wa}
Recall that $\{\mu\}$ is a geometric conjugacy class of a minuscule cocharacter $\mu: \mathbb{G}_m\rightarrow G_{\bar F}$. After choosing a suitable representative in $\{\mu\}$, we may assume $\mu\in X_*(T)^+$ via inner twisting, where $^+$ stands for the dominant cocharacters. We consider the adic space $\mathcal{F}(G, \mu)$ associated to the flag variety over $\mathrm{Spa}(\breve E)$. For $b\in G(\breve F)$, Rapoport and Zink has defined a weakly admissible locus
\[\mathcal{F}(G, \mu, b)^{wa}\subseteq \mathcal{F}(G, \mu)\] associated to $(G, \mu, b)$. Now we recall its definition.

Let $L|\breve F$ be a complete field extension. For any $x\in \mathcal{F}(G, \mu)(L)$, we can associate a cocharacter $\mu_{x}\in \{\mu\}$ defined over $L$. Let $\varphi-\mathrm{FilMod}_{L|\breve F}$ be the category of filtered isocrystals over $L|\breve F$. There is a functor
\[\begin{split}\mathcal{I}_{b, x}: \mathrm{Rep} G &\longrightarrow \varphi-\mathrm{FilMod}_{L/\breve F}\\ (V,\rho) &\mapsto (V_{\breve F}, \rho(b)\sigma, \mathrm{Fil}^{\bullet}_{\rho\circ \mu_x}V_{L})
\end{split}\]

The pair $(b, x)$ is called weakly admissible if for any $(V, \rho)\in \mathrm{Rep} G$, the filtered isocrystal $\mathcal{I}_{b, x}(V, \rho)$ is weakly admissible in the sense of Fontaine. More precisely, a filtered isocrystal $\mathcal{V}=(V, \varphi, \mathrm{Fil}^{\bullet}V_L)$ over $L|\breve F$ is called weakly admissible if for any subobject $\mathcal{V}'$ of $\mathcal{V}$ with filtration induced from $\mathcal{V}$, we have
\[t_H(\mathcal{V})=t_N(\mathcal{V}) \text{ and } t_H(\mathcal{V}')\leq t_N(\mathcal{V}'),\] where
$t_N(\mathcal{V})$ the $\pi_F$-adic valuation of $\det\varphi$ and
\[t_H(\mathcal{V}):=\sum_{i\in\Z} i\cdot \dim_L \mathrm{gr}^i_{\mathrm{Fil}^{\bullet}}(V_L).\]

 Let
\[\mathcal{F}(G, \mu, b)^{wa}(L):=\{x\in \mathcal{F}(G, \mu)(L)| (b, x) \text{ is weakly admissible}\}.\]

This defines the weakly admissible locus $\mathcal{F}(G, \mu, b)^{wa}$ which is a partially proper open subspace inside $\mathcal{F}(G, \mu)$ by \cite{RZ} Proposition 1.36.

\begin{remark}\label{remark_reduction wa to qs}\begin{enumerate}\item Let $b, b'\in G(\breve F)$ with  $[b]=[b']\in B(G)$,  then $\mathcal{F}(G, \mu, b)^{wa}\simeq \mathcal{F}(G, \mu, b)^{wa}$.
\item By \cite{RV} Proposition 3.1,  $\mathcal{F}(G, \mu, b)^{wa}$ is non-empty if and only if $[b]\in A(G,\mu)$.
\item Suppose the Frobenius maps on $H(\breve F)$ maps to $g^{-1}\sigma(-)g$ via the inner twisting $H_{\breve F}\stackrel{\sim}{\rightarrow} G_{\breve F}$ with $g\in G(\breve F)$. We have a bijection $B(G)\stackrel{\sim}{\rightarrow} B(H)$ which maps $[b]$ to $[b^H]$ where $b^H$ maps to $bg\in G(\breve F)$ via the inner twisting. By \cite{DOR} Proposition 9.5.3,  there is an identification \[ \mathcal{F}(G, \mu, b)^{wa}=\mathcal{F}(H, \mu, b^H)^{wa}.\]
    Therefore for the study of weakly admissible locus, it suffices to reduce to the quasi-split case.
\end{enumerate}
\end{remark}

In the following proposition, we will use the modification of $G$-bundles on the curve $X$ to give an equivalent definition of the weakly admissibility of a pair $(b, x)$ when $G$ is quasi-split.

\begin{proposition}\label{proposition_wa}[\cite{CFS} Proposition 2.7]
Assume that $G$ is quasi-split. Let $[b]\in A(G, \mu)$ and $x\in \mathcal{F}(G,\mu)(L)$. Then the pair $(b, x)$ is weakly admissible if and only if for any standard parabolic $P$ with Levi component $M$, any reduction $b_M$ of $b$ to $M$, and any $\chi\in X^\ast(P/Z_G)^+$ where $Z_G$ is the center of $G$, we have \[\deg\chi_\ast (\E_{b,x})_P\leq 0,\] where $(\E_{b, x})_P$ is the reduction to $P$ of $\E_{b, x}$ induced by the reduction $\E_{b_M}\times^{M}P$ of $\E_b$ by Lemma \ref{lemma_redction modfication}.
\end{proposition}


\subsection{Admissible locus}Rapoport and Zink has conjectured in \cite{RZ} the existence of an open subspace
\[ \Fc(G, \mu, b)^a\subseteq \Fc(G, \mu, b)^{wa} \] with a \'etale-$G$-local system $\mathcal{L}$ on $\Fc(G, \mu, b)^a$ such that these two spaces have the same classical points and the $G$-local system $\mathcal{L}$ interpolates a family of crystalline representations with value in $G(F)$.

When the local Shimura datum $(G,\mu, b)$ corresponds to a Rapoport-Zink space $\mathcal{M}(G, \mu, b)$ over $\breve E$, then the admissible locus $\Fc(G, \mu, b)^a$ is the image of the $p$-adic period mapping (\cite{RZ} Chapter 5)
\[\breve{\pi}: \mathcal{M}(G, \mu, b)\rightarrow \Fc(G, \mu), \] and the $G$-local system $\mathcal{L}$ corresponds to the Tate module of the universal $p$-divisble group with $G$-structures by descent via the $p$-adic period mapping.

For the general local Shimura datum $(G,\mu, b)$, the existence of the admissible locus is due to the work of Fargues-Fontaine \cite{FF}, Kedlaya-Liu \cite{KL} and Scholze \cite{Sch1}.

\begin{definition} Let $\Fc(G,\mu, b)^{a}$ be a subspace of $\Fc(G, \mu)$ stable under generalization with $C$-points defined as follows:
\[\Fc(G, \mu, b)^{a}(C)=\{x\in\Fc(G, \mu)(C)| \nu_{\E_{b, x}} \text{ is trivial }\}\] for any complete algebraically closed field $C$ over $F$.
\end{definition}

\begin{remark}
\begin{enumerate}
\item $\Fc(G, \mu, b)^{a}$ is an open subset of $\Fc(G, \mu)$ (\cite{KL}), and by definition
\[\Fc(G, \mu, b)^{a}\subset\Fc(G, \mu, b)^{wa}.\]Moreover,
\[\Fc(G, \mu, b)^{a}(K)=\Fc(G, \mu, b)^{wa}(K)\] for any finite extension $K$ over $\breve E$ (\cite{Ra2} A.5, \cite{CoFo}). In particular, $\Fc(G, \mu, b)^{a}\neq \emptyset$ if and only if $[b]\in A(G, \mu)$.
\item For $[b]\in A(G, \mu)$, the admissible locus $\Fc(G, \mu, b)^a$ coincides with the image of the $p$-adic period mapping from the local Shimura variety attached to $(G, \mu,b)$ to the flag variety $\Fc(G,\mu)$ (\cite{Sch}, \cite{SW}). It also coincides with the construction of admissible locus of Hartl \cite{Har} and Faltings \cite{Fal} when $(G, \mu, b)$ is a Hodge type local Shimura datum.
\item Via the bijection $B(G)\stackrel{\sim}{\rightarrow} B(H)$ by inner twisting which maps $[b]$ to $[b^H]$, there is an identification \[ \mathcal{F}(G, \mu, b)^{a}=\mathcal{F}(H, \mu, b^H)^{a}.\]
     Therefore for the study of admissible locus, we can also reduce to the quasi-split case.
\end{enumerate}
\end{remark}

For any $[b]\in B(G)$, consider the Newton stratification:
\[\Fc(G, \mu)=\coprod_{[b']\in B(G)} \Fc(G, \mu, b)^{[b']}\]
where $\Fc(G, \mu, b)^{[b']}$ is a subspace of $\Fc(G, \mu)$ stable under generalization with $C$-points defined by  \[\Fc^{[b']}(C)=\{x\in \Fc(C)|\E_{b,x}\simeq \E_{b'}\}\] for any complete algebraically closed field $C$ over $F$. Note that when $[b]$ is basic, we know precisely which strata show up in the Newton stratification by Proposition \ref{prop:description modification basique}. But it's unknown for non-basic $[b]$. Each stratum in the Newton stratification is locally closed by Kedlaya-Liu \cite{KL}. And it's clear that \[\Fc(G, \mu, b)^{[1]}=\Fc(G, \mu, b)^a.\]

\begin{remark}The $\tilde{J}_b(K)$-action on $\Fc(G, \mu)(C)$ induces an action on each stratum $\Fc(G, \mu, b)^{[b']}(C)$. In particular, $\tilde{J}_b(K)$ acts on $\Fc(G, \mu,b)^a(C)$.
\end{remark}

\section{Hodge-Newton-decomposability}\label{section_HN-decomposable}
Let $M_b=\mathrm{Cent}_H([\nu_b])$ be the centralizer of $[\nu_b]$.

\begin{definition}\begin{enumerate}
\item A triple $(G, \mu, b)$ (resp. $(G, \delta, b)$ with $\delta\in X_*(A)^+_{\Q}$) is called Hodge-Newton-decomposable (or HN-decomposable for short) if $[b]\in A(G,\mu)$ (resp. $[b]\in B(G, \e, \delta)$ with $\e=\kappa_G(b)\in \pi_1(G)_{\Gamma}$) and there exists a strict standard Levi subgroup $M$ of the quasi-split inner form $H$ of $G$ containing $M_b$, such that $\mu^{\diamond}-[\nu_b]\in\langle\Phi^\vee_{0,M}\rangle_{\Q}$ (resp. $\delta-[\nu_b]\in\langle\Phi^\vee_{0,M}\rangle_{\Q}$). Otherwise, the triple $(G, \mu, b)$ (resp. $(G, \delta, b)$) is called Hodge-Newton-indecomposable (or HN-indecomposable for short).
\item A pair $(G, \mu)$ is called fully Hodge-Newton decomposable (or fully HN-decomposable for short) if for any non-basic $[b]\in B(G,\mu)$, the triple $(G, \mu, b)$ is HN-decomposable. In this case, we also say the Kottwitz set $B(G,\mu)$ is fully Hodge-Newton decomposable.
\item The generalized Kottwtiz set $B(G, \e, \delta)$ is called fully HN-deomposable, if for any non-basic $[b]\in B(G,\e, \delta)$, the triple $(G, \delta, b)$ is HN-decomposable.
\end{enumerate}
\end{definition}

\begin{remark}The notion of fully Hodge-Newton decomposability is first introduced and systematically studied by G\"ortz, He and Nie in \cite{GoHeNi}. They give equivalent conditions for a pair $(G, \mu)$ to be fully HN-decomposable and classify all such pairs.
\end{remark}

In the quasi-split case, we have the following equivalent definition for the HN-decomposability.
\begin{lemma}\label{lemma_HN-decomp}[comp. \cite{CFS} Lemma 4.11]Suppose $G$ is quasi-split. Let $[b]\in B(G, \mu^{\sharp}+\e, \mu^{\diamond})$ for some $\e\in \pi_1(G)_{\Gamma, tor}$. Then the following three conditions are equivalent:
\begin{enumerate}
\item the triple $(G, \mu, b)$ is HN-decomposable,
\item there exist a strict standard Levi subgroup $M$ containing $M_b$ and a unique element $\e_M\in \pi_1(M)_{\Gamma, tor}$ such that $[b_M]\in B(M,\mu^{\sharp}+\e_M, \mu^{\diamond})$ and $\e_M$ maps to $\e$ via the natural map $\pi_1(M)_{\Gamma}\rightarrow \pi_1(G)_{\Gamma}$, where $b_M$ is the reduction of $b$ to $M$ deduced from its canonical reduction to $M_b$ combined with the inclusion $M_b\subseteq M$,
\item there exist a strict standard Levi subgroup $M$ containing $w_0M_bw_0^{-1}$ and a unique element $\tilde \e_M\in \pi_1(M)_{\Gamma, tor}$ such that $[\tilde{b}_M]\in B(M,(\tilde{w}_0\mu)^{\sharp}+\tilde{\e}_M, (\tilde{w}_0\mu)^{\diamond})$ and $w_0\tilde{\e}_M$ maps to $\e$ via the natural map $\pi_1(M)_{\Gamma}\rightarrow \pi_1(G)_{\Gamma}$,  where $\tilde{b}_M$ is the reduction of $b$ to $M$ deduced from its canonical reduction $w_0b_{M_b}w_0^{-1}$ to $w_0M_bw_0^{-1}$ combined with the inclusion $w_0M_bw_0^{-1}\subseteq M$ and $\tilde w_0\mu=(w_0\mu)_{M-dom}$ is the $M$-dominant representative in $W_Mw_0\mu$.
\end{enumerate}
\end{lemma}
\begin{proof}The proof of the equivalence of (1) and (2) is similar as that of \cite{CFS} Lemma 4.11. Indeed, by definition, $(G, \mu, b)$ is HN-decomposable if and only if $[b_M]\in A(M, \mu)$ for some strict standard Levi subgroup of $G$. By Remark \ref{AGmu}, we may assume $[b_M]\in B(M,\mu^{\sharp}+\e_M, \mu^{\diamond})$ for some $\e_M\in \pi_1(M)_{\Gamma, tor}$ which maps to $\e\in \pi_1(G)_{\Gamma}$. The uniqueness of $\e_M$ follows from Lemma \ref{lemma_BGmu reduction Levi}(1).

The equivalence between (2) and (3) is due to the fact that there is a bijection between $B(M,\mu^{\sharp}+\e_M, \mu^{\diamond})$ and $B(w_0Mw_0^{-1}, (\tilde{w}_0\mu)^{\sharp}+w_0\e_M, (\tilde{w}_0\mu)^{\diamond})$ induced by the conjugation by $w_0$. \end{proof}

\begin{corollary}\label{coro_HN-indecomp}Suppose $G$ is quasi-split. Let $[b]\in A(G, \mu)$. Then the following three conditions are equivalent:
\begin{enumerate}
\item the triple $(G, \mu, b)$ is HN-indecomposable,

\item  $\mu^{\sharp}-\kappa_{M_b}(b_{M_b})\in \pi_1(M_b)_{\Gamma}\otimes\Q$ is a linear combination of \[\{\alpha^{\vee,\sharp}\in \pi_1(M_b)_{\Gamma}\otimes \Q|\alpha\in\Delta_0, \langle\alpha, [\nu_b]\rangle >0\}\] with coefficients of positive integers,
\item for all $\alpha\in\Delta_0$, $\langle\alpha, [\nu_b]\rangle >0$ implies the coefficient of $\alpha^\vee$ in $\mu^{\diamond}-[\nu_b]$ is positive.
\end{enumerate}
\end{corollary}

\begin{corollary}\label{coro-HN-decomp}Suppose $G$ is quasi-split. Suppose $[b]\in B(G, \mu^{\sharp}+\e, \mu^{\diamond})$ for some $\e\in \pi_1(G)_{\Gamma, tor}$ and $(G, \mu, b)$ is HN-decomposable. Then there exists a unique strict standard Levi subgroup $M$, a unique element $\e_M\in \pi_1(M)_{\Gamma, tor}$ and a reduction $\tilde{b}_M$ such that \begin{enumerate}
\item $\e$ is the image of $w_0\e_M$ under the natural map $\pi_1(M)_{\Gamma}\rightarrow \pi_1(G)_{\Gamma}$,

\item $[\tilde{b}_M]\in B(M, (\tilde{w}_0\mu)^{\sharp}+\e_M,(\tilde{w}_0\mu)^{\diamond})$,

\item $(M, \tilde{w}_0\mu, \tilde{b}_M)$ is HN-indecomposable.
\end{enumerate}
\end{corollary}

\begin{proof}Let $M_1$ be the standard Levi subgroup of $G$ such that its simple roots are described as follows: \[\Delta_{M_1,0}=\Delta_{M_b, 0}\cup \{\alpha\in \Delta_0| n_{\alpha}>0\},\]
where $\mu^{\diamond}-[\nu_b]=\sum_{\alpha\in \Delta_0}n_{\alpha}\alpha^{\vee}$ with $n_{\alpha}\geq 0$. Let $b_{M_1}$ be the reduction of $b$ to $M_1$ deduced from its canonical reduction to $M_b$ combined with the inclusion $M_b\subseteq M_1$. Then it's easy to see $(M_1, \mu, b_{M_1})$ is HN-indecomposable and $M_1$ is the unique standard Levi subgroup with this property. Let $M=w_0M_1w_0^{-1}$. The rest of the assertions can be proved in the same way as Lemma \ref{lemma_HN-decomp}.
\end{proof}

\begin{remark}\label{remark_eM}Notations as in Corollary \ref{coro-HN-decomp}, as $w_0\e$ is the image of $\e_M$ under the natural injective map $\pi_1(M)_{\Gamma}\rightarrow \pi_1(G)_{\Gamma}$, we may identify $\e_M$ and $w_0\e$.
\end{remark}

The following proposition is a key ingredient to the proof of the main result.

\begin{proposition}\label{Prop_preimage_admissible} Suppose $G$ is quasi-split. Let $[b]\in A(G, \mu)$ such that $(G,\mu, b)$ is HN-decomposable. Let $M\subseteq G$ be the strict standard Levi subgroup such that $(M, w_0\mu, \tilde{b}_M)$ is HN-indecomposable as in Corollary \ref{coro-HN-decomp}. Then \[\mathrm{pr}_{\tilde w_0}^{-1}(\Fc(M, \tilde w_0\mu, \tilde{b}_M)^a(C))=\Fc(G, \mu, b)^a(C),\]
\[\mathrm{pr}_{\tilde w_0}^{-1}(\Fc(M, \tilde w_0\mu, \tilde{b}_M)^{wa}(C))=\Fc(G, \mu, b)^{wa}(C).\]
\end{proposition}

\begin{remark}The appearence of $\tilde{w}_0$ in the statement is due to the fact that $\nu_{\E_b}=-w_0[\nu_b]$ by Proposition \ref{prop_invariant_bundle}.
\end{remark}
The remaining of this section is devoted to the proof of this proposition. Suppose $(G, \mu, b)$ is HN-decomposable with $G$ quasi-split. Suppose $M$ is the Levi subgroup of $G$ such that $(M, \tilde{w}_0\mu, \tilde{b}_M)$ is HN-indecomposable as in Corollary \ref{coro-HN-decomp}. Let $P$ be the standard parabolic subgroup of $G$ with $M$ as Levi component.

\begin{lemma}\label{lemma: adm in Shubert}
$\Fc(G, \mu, b)^{wa}(C)\subseteq P(C)w_0 P_{\mu}(C)/P_{\mu}(C)$.
\end{lemma}
\begin{proof}
Suppose $\Fc(G, \mu, b)^{wa}(C)\cap P(C)wP_{\mu}(C)/P_{\mu}(C)\neq \emptyset$ for some $w\in W$. We want to show $w_0\in W_{P}wW_{P_{\mu}}$.  For any $x\in \Fc(G, \mu, b)^{wa}(C)\cap P(C)wP_{\mu}(C)/P_{\mu}(C)$, we have \[(\E_{b, x})_P\simeq\E_{\tilde b_M, \mathrm{pr}_w(x)}\] by Lemma \ref{lemma:reduction to P modification}. The weak semi-stability of $\E_{b,x}$ implies that   $\deg\chi_*(\E_{\tilde{b}_M, \mathrm{pr}_w(x)})\leq 0$ for any $\chi\in X^*(M/Z_G)^+$. On the other side, using the fact that $[\tilde{b}_M]\in A(M, \tilde{w}_0\mu)$, we have \[c_1^M(\E_{\tilde{b}_M, \mathrm{pr}_w(x)})=(w\mu)^\sharp-\kappa_M(\tilde{b}_M)=(w\mu)^{\sharp}-(w_0\mu)^{\sharp} \text{ in }\pi_1(M)_{\Gamma}\otimes\Q,\] and $\deg\chi_*(\E_{b_M, \mathrm{pr}_w(x)})=\langle w\mu- w_0\mu, \chi\rangle\leq 0$ for any $\chi\in X^*(M/Z_G)^+$ by Proposition \ref{prop_invariant_bundle}. Therefore the equality holds for any $\chi$ and $(w\mu)^{\sharp}=(w_0\mu)^{\sharp}$ in $\pi_1(M)_{\Gamma}$. On the other hand, as $w\mu\geq w_0\mu$, $w\mu$ and $w_0\mu$ has the same image in $\pi_1(M)$. The result follows.
\end{proof}

Notations as in Corollary \ref{coro-HN-decomp}. Let $\E^G_{\e}$ be the $G$-bundle such that $\nu_{\E^G_{\e}}$ is trivial and $c_1^G(\E^G_{\e})=-\e$. Similarly, let $\E^M_{w_0\e}$ be the $M$-bundle such that $\nu_{\E^M_{w_0\e}}$ is trivial and $c_1^M(\E^M_{w_0\e})=-w_0\e$ (cf. Remark \ref{remark_eM}). Let $\E^P_{w_0\e}:=\E^M_{w_0\e}\times^M P$. When $[b]\in B(G, \mu)$ (i.e., $\e=0$), we have $\E^G_{\e}=\E^G_1$, $\E^{M}_{w_0\e}=\E^M_1$ and $\E^P_{w_0\e}=\E^P_1$ are the trivial bundles.

Consider the following commutative diagram (cf. \cite{Man} \cite{Sh}) of de Rham period maps for different groups from local Shimura varieties at infinite level to flag varieties.
\[\xymatrix{& \M(P, \tilde w_0\mu, \tilde{b}_P)_{\infty}\ar[dl]_{\xi_M}\ar[dr]^{\xi_G}
\ar@{->>}[dd]^{\pi_{dR}^P} &\\
 \M(M, \tilde w_0\mu, \tilde{b}_M)_{\infty}\ar@{->>}[dd]_{\pi_{dR}^M} & &\M(G, \mu, b)_{\infty}\ar@{->>}[dd]^{\pi_{dR}^G}\\
  &\Fc(P, \tilde w_0\mu, \tilde{b}_P)^a\ar[dr]_{\tilde\xi_G}\ar[dl]^{\tilde \xi_M}&\\
  \Fc(M, \tilde w_0\mu, \tilde{b}_M)^a & &\Fc(G, \mu, b)^a }\]
where
\begin{itemize}
\item $\tilde{b}_P$ is the reduction of $b$ to $P$ induced by the reduction $\tilde{b}_M$ of $b$ to $M$ combined with the inclusion $M\subseteq P$.
\item $\M(G,\mu, b)_{\infty}$ (resp. $\M(M, \tilde w_0\mu, \tilde{b}_M)$, resp. $\M(P, \tilde w_0\mu, \tilde{b}_P)$) classifies modifications of type $\mu$ (resp. $\tilde w_0\mu$, resp. $\tilde{w}_0\mu$) between $\E^G_b$ (resp. $\E^M_b$, resp. $\E^P_b$) and $\E^G_\e$ (resp. $\E^M_{w_0\e}$, resp. $\E^P_{w_0\e}$).
\item $\pi^G_{dR}$, $\pi^M_{dR}$, $\pi^P_{dR}$ are the de Rham period maps. More precisely, for a modification in $\M(G, \mu, b)_{\infty}$ its image by $\pi^G_{dR}$ is $x$ if $\E_{b, x}^G=\E_{\e}^G$. Similarly for $\pi^M_{dR}$ and $\pi^P_{dR}$. We define $\Fc(P, \tilde{w}_0\mu, \tilde{b}_P)^a$ to be the image of $\pi_{dR}^P$.
\item $\xi_G$ (resp. $\xi_M$) is the induced modification via the natural morphism $P\ra G$ (resp. the projection to the Levi quotient $P\ra M$).
\item $\tilde\xi_G$ is induced from
\begin{eqnarray}\label{eqn_flag variet P to G}\Fc(P, \tilde w_0\mu)=P/P_{\tilde w_0\mu}\cap P&\rightarrow& P\tilde w_0 P_{\mu}/P_{\mu}\subseteq G/P_{\mu}=\Fc(G, \mu)
\\ \nonumber aP_{\tilde{w}_0\mu}\cap P &\mapsto &a(P_{\tilde{w}_0\mu}\cap P)\tilde{w}_0P_{\mu}=a\tilde{w}_0P_{\mu}
\end{eqnarray}

\item $\tilde\xi_M$ is induced from the natural projection $P\rightarrow M$ to the Levi component
\[
\Fc(P, \tilde w_0\mu)=P/P_{\tilde w_0\mu}\rightarrow M/M_{\tilde w_0\mu}=\Fc(M, \tilde w_0\mu) \]
\end{itemize}

By the proof of \cite{CFS} Lemma 6.3, we have the following fact.
\begin{lemma}\label{lemma_reduction two parabolic}Let $\E$ be a $G$-bundle and let $P'\subseteq P$ be a standard parabolic subgroup of $G$ contained in $P$. There is a bijection between
\begin{itemize}
\item reductions $\E_{P'}$ of $\E$ to $P'$,
\item reductions $\E_P$ to $P$ together with a reduction $(\E_P\times^P M)_{M\cap P'}$ of $\E_P\times^P M$ of $M\cap P'$.
\end{itemize}
    Moreover, this bijection indentifies $\E_{P'}\times^{P'}M'$ and $(\E_P\times^P M)_{M\cap P'}\times^{M\cap P'}M'$, where $M'$ is the Levi component of $P'$.
\end{lemma}
\bigskip
Now we can prove the following result.
\begin{lemma}\label{Lemma_tilde_varphi_isom} In the above diagram, we have $\tilde\xi_G$ is an isomorphism of adic spaces and
 \[\M(G, \mu, b)_{\infty}\simeq\coprod_{\mathrm{Aut}(\E^G_{\e})/\mathrm{Aut}(\E^P_{w_0\e})}\M(P, \tilde w_0\mu, b)_{\infty}.\]
\end{lemma}
\begin{proof} We only deal with the case $[b]\in B(G, \mu)$. The proof for $[b]\in A(G, \mu)$ is similar. For the first assertion, as $\Fc(P, \tilde w_0\mu)\rightarrow\Fc(G, \mu)$ is an open immersion, it suffices to show \[\Fc(P, \tilde w_0\mu, \tilde{b}_P)^a(C)\rightarrow \Fc(G,\mu,b)^a(C)\] is surjective for any complete algebraic closed field $C$. For any $x\in \Fc(G, \mu, b)^a(C)$, by Lemma \ref{lemma: adm in Shubert}, $x\in P(C)\omega_0P_{\mu}(C)/P_{\mu}(C)$. Let $(\E_{b, x})_P$ be the reduction of $\E_{b, x}$ to $P$ induced by the reduction $\E_{\tilde{b}_P}$ of $\E_b$ to $P$. Write $(\E_{b, x})_P=\E_{\tilde{b}_P, y}$ where $y$ is the preimage of $x$ via (\ref{eqn_flag variet P to G}). We want to show $y\in \Fc(P, \tilde{w}_0\mu, \tilde{b}_P)^a(C)$. It's equivalent to show $(\E_{b, x})_P=\E_{\tilde{b}_P, y}$ is a trivial $G$-bundle. The isomorphism of $G$-bundles $\E_{b,x}\simeq \E_1$ induces a reduction $(\E_1)_P$ of $\E_1$ to $P$ and an isomorphism of $P$-bundles. $(\E_{b,x})_P\simeq (\E_1)_P$. We want to show that $(\E_1)_P$ is a trivial $P$-bundle.  By corollary \ref{Coro_HN equality among reduction to M}, it suffices to show that $\E_M:=(\E_1)_P\times^P M$ is a trivial $M$-bundle. We first show that
\begin{eqnarray}\label{eqn_c_1(M)=0}c_1^M(\E_M)=0.\end{eqnarray} We have the following equalities
\begin{eqnarray*}c_1^M((\E_M)&=&c_1^M((\E_{b, x})_P\times^P M)\\
&\stackrel{\text{Lemma} \ref{lemma:reduction to P modification}}{=}&c_1^M(\E_{\tilde{b}_M, \mathrm{pr}_{w_0}(x)})\\
&=&(w_0\mu)^{\sharp}-\kappa_M(\tilde{b}_M)\\
&=&0\in \pi_1(M)_\Gamma,\end{eqnarray*} where the last equality holds because $[\tilde{b}_M]\in B(M, \tilde{w}_0\mu)$. Now it remains to show the slope $\nu_{\E_M}=0$. Let $(\E_M)_{P'_M}$ be the canonical reduction of $\E_M$ to a standard parabolic subgroup $P'_M$ of $M$. Write $P'=P'_M\cdot R_u(P)$, where $R_u(P)$ denotes the unipotent radical of $P$. Note that $P'$ is a standard parabolic subgroup of $G$ with $P'\cap M=P'_M$. Let $M'$ be the Levi component of $P'_M$. By \cite{F3} Proposition 5.16,
\[\nu_{(\E_M)_{P'_M}\times^{P'_M}M'}=\nu_{\E_M}\in X_*(A)_{\Q}.\] By Lemma \ref{lemma_reduction two parabolic}, $(\E_M)_{P'_M}$ corresponds to a reduction of $\E_1$ to $P'$. Hence the semi-stability of $\E_1$ implies that
\[\langle\chi, \nu_{\E_M}\rangle\leq 0, \ \forall \chi\in X^*(M'/Z_G)^+.\] On the other hand, the equality (\ref{eqn_c_1(M)=0}) implies that $\nu_{\E_M}$ is a non-negative linear combination of simple coroots in $M$. Hence
\[\langle\chi, \nu_{\E_M}\rangle\geq 0, \ \forall \chi\in X^*(M'/Z_G)^+.\] It follows that $\nu_{\E_M}=0$.

The second assertion is obtained from the first assertion and the fact that $\M (G, \mu, b)_{\infty}$ (resp. $\M (P, \tilde{w}_0\mu, \tilde{b}_P)_\infty$) is the $\underline{\mathrm{Aut}(\E^G_{\e})}$-torsor (resp. $\underline{\mathrm{Aut}(\E^P_{w_0\e})}$-torsor) over $\Fc(G, \mu, b)^a$ (resp. $\Fc(P, \tilde{w}_0\mu, \tilde{b}_P)$).
\end{proof}
\begin{remark}This lemma was proved for unramified local Shimura datum of PEL type by Mantovan \cite{Man} section 8.2 and Shen \cite{Sh} corollary 6.4.
\end{remark}

In order to prove Proposition \ref{Prop_preimage_admissible}, we also need the following Lemmas.

\begin{lemma}\label{lemma_wa Levi in wa G}Suppose $G$ is quasi-split. We consider $\Fc(M, w_0\mu)$ as a subspace of $\Fc(G, \mu)$ via the natural injective morphism $\Fc(M, \tilde{w}_0\mu)\rightarrow \Fc(G, \mu)$. Then we have
\begin{eqnarray*}\Fc(M, \tilde{w}_0\mu, \tilde{b}_M)^a&\subseteq& \Fc(G, \mu, b)^a,\\
\Fc(M, \tilde{w}_0\mu, \tilde{b}_M)^{wa}&\subseteq& \Fc(G, \mu, b)^{wa}.\end{eqnarray*}
\end{lemma}
\begin{proof}The first assertion is clear. For the second one, let $x\in \Fc(M, \tilde{w}_0\mu, \tilde{b}_M)^{wa}(C)$ and let $x_G$ be its image via the natural morphism $\Fc(M, \tilde{w}_0\mu)\rightarrow \Fc(G, \mu)$.  Suppose $Q$ is a standard parabolic subgroup with Levi component $M_Q$. Let $b_{M_Q}$ be a reduction of $b$ to $M_Q$. Let $w\in W^{\Gamma}$ and $b_1\in M_1(\breve F)$ with $M_1=M\cap w^{-1}M_Qw$ as in Lemma \ref{lemma_reduction to smaller Levi}. Let $Q_1$ be the standard parabolic subgroup of $M$ with Levi component $M_1$. Since
\[(\E_{b_1}\times^{M_1}Q_1)\times^{Q_1, \mathrm{ad}(w)} Q\simeq \E_{b_{M_Q}}\times^{M_Q}Q \] by Lemma \ref{lemma_reduction to smaller Levi}, we have \[(\E_{b,x_G})_Q\simeq(\E_{\tilde{b}_M, x})_{Q_1}\times^{Q_1}Q,\]where $Q_1\rightarrow Q$ is induced by the $\mathrm{ad}{\dot w}$ and where $(\E_{b,x_G})_Q$ (resp.$(\E_{\tilde{b}_M, x})_{Q_1}$) is the reduction of $\E_{b,x_G}$ (resp. $\E_{\tilde{b}_{M}, x}$) to $Q$ (resp. $Q_1$) induced by the reduction $\E_{b_{M_Q}}\times^{M_Q} Q$ (resp. $\E_{b_1}\times^{M_1} Q_1$) of $\E_b$ (resp. $\E_{\tilde{b}_M}$) to $Q$ (resp. $Q_1$). For any $\chi\in X^*(Q/Z_G)^+$, we have $\deg\chi_*(\E_{b,x_G})_Q=\deg\chi'_*(\E_{\tilde{b}_M, x})_{Q_1}$ with $\chi'=\chi\circ \mathrm{ad}{\dot w}\in X^*(Q_1)$. Write $\chi'=\chi'_1+\chi'_2$ with $\chi'_1=\mathrm{Av}_{W_M}(\chi')$ the $W_M$ average of $\chi'$. Up to replacing $\chi$ by a multiple $m\chi$ for $m\in\mathbb{N}$ big enough, we may assume $\chi'_1\in (\Phi_M^\vee)^{\perp}=X^*(M)$, $\chi'_2\in X^*(M_1/Z_M)$. Moreover, the choice of $w$ implies $w\beta\in \Phi_G^+$ for any $\beta\in\Phi_M^+$. Hence $\chi'_2\in X^*(M_1/Z_M)^+$, here $^+$ stands for $M$-dominant. Then
\begin{eqnarray*}\deg\chi_*(\E_{b,x})_Q&=&\deg\chi'_*(\E_{\tilde{b}_M, x})_{Q_1}\\&=&\underbrace{\deg\chi'_{1*}(\E_{\tilde{b}_M, x})}_{=0}+\underbrace{\deg\chi'_{2*}(\E_{\tilde{b}_M, x})_{Q_1}}_{\leq 0}\leq 0\end{eqnarray*}
since $[\tilde{b}_M]\in B(M, \tilde {w}_0\mu)$ combined with the weak admissibility of $x$.
\end{proof}

Recall that $M$ and $\tilde{b}_M$ are as in Corollary \ref{coro-HN-decomp}. The following lemma reflects the fact that the category of isocrystals is semi-simple. If we can decompose an isocrystal in two different ways, then we can decompose it in a way that is finer than the previous two decompositions.

\begin{lemma}\label{lemma_reduction to smaller Levi}Suppose $G$ is quasi-split. Let $b_{M_Q}$ be a reduction of $b$ to $M_Q$, where $M_Q$ is the Levi component of a standard parabolic subgroup $Q$. Then there exist $w\in W^{\Gamma}$ (where $W^{\Gamma}$ can be identified with the relative Weyl group of $G$), $g_1\in M(\breve F)$, $g'_1\in (w^{-1}M_{Q}w)(\breve F)$ and $b_1\in M_1(\breve F)$ with $M_1=w^{-1}M_{Q}w\cap M$ such that
\begin{enumerate}
 \item $w$ is the minimal length element in $W_{M_Q}^{\Gamma}wW_M^{\Gamma}$;
 \item $(b_1, g_1)$ is a reduction of $\tilde{b}_M$ to $M_1$;
 \item $(b_1, g'_1)$ is a reduction of $\dot{w}^{-1}b_{M_Q}\sigma(\dot{w})$ to $M_1$ with $\dot{w}\in N(T)(\breve F)$ a representative of $w$.
 \end{enumerate}Moreover, $W_{M_Q}wW_M$ is the generic relative position between the reduction $\E_{\tilde{b}_M}\times^M P$ of $\E_b$ to $P$ and the reduction $\E_{b_{M_Q}}\times^{M_Q}Q$ of $\E_b$ to Q.
\end{lemma}
\begin{proof} The last assertion is implied by the conditions (2) and (3). We first show that there exists elements satisfying condition (2) and (3). More precisely, we claim that there exists $\tilde{w}\in W^{\Gamma}$, $\tilde g_1\in M(\breve F)$, $\tilde{g}'_1\in \tilde{w}^{-1}M_{Q}\tilde{w}(\breve F)$ and $\tilde{b}_1\in \tilde{M}_1(L)$ with $\tilde{M}_1=\tilde{w}^{-1}M_{Q}\tilde{w}\cap M$ such that
\begin{enumerate}
 \item [(2)] $(\tilde{b}_1, \tilde{g}_1)$ is a reduction of $\tilde{b}_M$ to $\tilde{M}_1$;
 \item[(3)] $(\tilde{b}_1, \tilde{g}'_1)$ is a reduction of $\dot{\tilde{w}}^{-1}b_{M_Q}\sigma(\dot{\tilde{w}})$ to $\tilde{M}_1$ with $\dot{\tilde{w}}\in N(T)(\breve F)$ a representative of $\tilde{w}$.
 \end{enumerate}
Note that a representative $\dot{\tilde{w}}$ of $\tilde{w}$ can be chosen in $N(T)(\breve F)$ by Steinberg's theorem (\cite{SerreCohoGal} III 2.3).

By definition, $\tilde{b}_M$ is induced from $w_0b_{M_b}w_0^{-1}$ via the natural inclusion $w_0M_bw_0^{-1}\subseteq M$. Therefore, without loss of generality, we may assume $M=w_0M_bw_0^{-1}$. Since $b_{M_Q}$ is a reduction of $b$ to $M_Q$, Up to $\sigma$-conjugation, we may assume $\nu_{b_{M_Q}}:\mathbb{D}\rightarrow M_Q$ is defined over $F$ and has image in the split maximal torus $A$ which is $M_Q$-anti-dominant. Choose $\tilde{w}\in W^{\Gamma}$ such that $\tilde{w}^{-1}\nu_{b_{M_Q}}=w_0[\nu_b]$ is $G$-anti-dominant. Then $\mathrm{Cent}_{\tilde{w}^{-1}M_Q\tilde{w}}(\nu_{\dot{\tilde{w}}^{-1}b_{M_Q}\sigma(\dot{\tilde w})})=\tilde{w}^{-1} M_Q \tilde{w}\cap M=\tilde{M}_1$ and $\dot{\tilde{w}}^{-1}b_{M_Q}\sigma(\dot{\tilde w})$ has a canonical reduction $(\tilde{b}_1, \tilde{g}'_1)$ to $\tilde{M}_1$. Then (3) follows.

Let $\tilde{b}'_1$ be the image of $\tilde{b}_1$ via the inclusion $\tilde{M}_1\subseteq M$. For (2), it suffices to show $[\tilde{b}'_1]=[\tilde{b}_M]$ in $B(M)$. Clearly, \[ [\nu_{\tilde{b}'_1}]_M=[\nu_{\tilde{b}_M}]_M=(w_0[\nu_b])_{M-dom}\in \N (M).\] By the injectivity of the map (\ref{eqn_invariants Kottwitz})
\[(\nu, \kappa_M): B(M)\rightarrow \mathcal{N}(M)\times \pi_1(M)_{\Gamma},\]it suffices to show $\kappa_M(\tilde{b}'_1)=\kappa_M(\tilde{b}_M)\in \pi_1(M)_{\Gamma}$. Since
\[\kappa_M(a)=\nu_a\text { in }\pi_1(M)_{\Gamma}\otimes_{\Z}\Q,\ \  \forall a\in M(\breve F),\]we have $\kappa_M(\tilde{b}'_1)-\kappa_M(\tilde{b}_M)\in\pi_1(M)_{\Gamma,tor}$. As $\tilde{b}_1$ and $\tilde{b}_M$ are both reductions of $b$, we have
   \[\kappa_G(\tilde{b}'_1)=\kappa_G(\tilde{b}_1)=\kappa_G(\tilde{b}_M)\text{ in }\pi_1(G)_{\Gamma}.\] The result follows by the injectivity of the map
   $\pi_1(M)_{\Gamma,tor}\rightarrow\pi_1(G)_{\Gamma, tor}$ (cf. proof of Lemma~\ref{lemma_HN-decomp}).

Let $w=w_1\tilde{w}w_2$ be the unique minimal length element in $W^{\Gamma}_{M_Q}\tilde{w}W_M^{\Gamma}$ with $w_1\in W_{M_Q}^{\Gamma}$ and $w_2\in W_M^{\Gamma}$. Then $w\in W^{\Gamma}$ and we can choose representatives $\dot w_2\in M(\breve F)$, $\dot w\in G(\breve F)$ of $w_2$ and $w$ respectively again by Steinberg's theorem. Let
\begin{eqnarray*}b_1&:=&\dot {w}_2^{-1}\tilde b_1 \sigma(\dot{w}_2)\in M_1(\breve F),\\
g_1&:=&\tilde{g}_1\dot{w}_2\in M(\breve F),\\
g'_1&:=&\dot {w}^{-1}\dot{\tilde{w}}\tilde g'_1 \dot{w}_2\in w^{-1}M_Qw (\breve F).\end{eqnarray*}They satisfy the desired properties (1)-(3).
\end{proof}

Now we are ready to prove the main result of this section.

\begin{proof}[Proof of proposition \ref{Prop_preimage_admissible}]We only deal with the case $[b]\in B(G, \mu)$. The proof for $[b]\in A(G, \mu)$ is similar. By Lemma \ref{lemma: adm in Shubert} and Lemma \ref{Lemma_tilde_varphi_isom}, we have the equality of morphisms: $\mathrm{pr}_{\tilde w_0}=\tilde\xi_M\circ {\tilde\xi_G}^{-1}$ and \[\Fc(G, \mu, b)^a\subseteq \mathrm{pr}_{w_0}^{-1}(\Fc(M, w_0\mu, \tilde{b}_M)^{a}).\] Now it remains to show that $\E_{b,x}$ is a trivial $G$-bundle for any $x\in \mathrm{pr}_{\tilde w_0}^{-1}(\Fc(M, \tilde w_0\mu, b)^a(C))$. By Lemma~\ref{lemma:reduction to P modification}, $(\E_{b,x})_P \times^P M \simeq \E_{\tilde{b}_M,\mathrm{pr}_{\tilde w_0} (x)}$ which is a trivial $M$-bundle. The result then follows by Corollary \ref{Coro_HN equality among reduction to M}.

The proof for the weak admissibility is similar. Suppose $x\in \Fc(G, \mu, b)^{wa}(C)$. Then $x\in P(C)w_0P_{\mu}(C)/P_{\mu}(C)$ by Lemma \ref{lemma: adm in Shubert}. We show that $\mathrm{pr}_{\tilde w_0}(x)\in\Fc(M, \tilde w_0\mu, \tilde{b}_M)^{wa}(C)$. Suppose $M'\subset M$ is a standard Levi subgroup of $M$ with $b'\in M'(L)$ a reduction of $\tilde{b}_M$ to $M'$. Let $P'_M\subseteq M$ (resp. $P'\subseteq G$) be the standard parabolic subgroup of $M$ (resp. $G$) with Levi component $M'$. We need to show that
\begin{eqnarray*}\deg \chi_*(\E_{\tilde{b}_M, \mathrm{pr}_{\tilde w_0}(x
)})_{P'_M}\leq 0, \forall \chi\in X^*(P'_M/Z_M)^+\end{eqnarray*} where $(\E_{\tilde{b}_M, \mathrm{pr}_{\tilde w_0}(x)})_{P'_M}$ is the reduction to $P'_M$ of $\E^M_{b'_{M}, \mathrm{pr}_{\tilde w_0}(x)}$ induced by the reduction $\E_{b'}\times^{M'} P'_M$ of $\E_{\tilde{b}_M}$. By Remark \ref{remark_vector v}, it's equivalent to show
\begin{eqnarray}\label{eqn_wa} \mathrm{Av}_{W_{M'}}(\nu_{\E_{M'}})\leq_M 0 \text{ in } X_*(A)_{\Q},\end{eqnarray}
 where $\E_{M'}:=(\E_{\tilde{b}_M, \mathrm{pr}_{w_0}(x)})_{P'_M}\times^{P'_M} M'$. By Theorem \ref{theoSchi},
\begin{eqnarray}\label{eqn_wa_Av1} \mathrm{Av}_{W_{M'}}(\nu_{\E_{M'}})\leq_M \nu_{\E_M}  \text{ in } X_*(A)_{\Q}, \end{eqnarray}
where $\E_{M}:=\E_{\tilde{b}_M, \mathrm{pr}_{w_0}(x)}$. On the other hand, by Lemma \ref{lemma:reduction to P modification} and Lemma \ref{lemma_reduction two parabolic}, we have
\[\E_{M'}=(\E_{\tilde{b}_M, \mathrm{pr}_{w_0}(x)})_{P'_M}\times^{P'_M} M'\simeq((\E_{b,x})_P\times^P M)_{P'_M}\times^{P'_M} M'\simeq (\E_{b,x})_{P'}\times^{P'}M',\]
 where  $(\E_{b, x})_{P'}$ is the reduction to $P'$ of $\E_{b, x}$ induced by the reduction $\E_{b'}\times^{M'} P'$ of $\E_b$. The weak semi-stability of $\E_{b,x}$ implies that
\begin{eqnarray}\label{eqn_wa_Av2}\mathrm{Av}_{W_{M'}}(\nu_{\E_{M'}})\leq_G 0 \text{ in } X_*(A)_{\Q}. \end{eqnarray} Hence the inequality (\ref{eqn_wa}) follows from (\ref{eqn_wa_Av1}) and (\ref{eqn_wa_Av2}) combined with the fact that \[\mathrm{Av}_{W_M}(\nu_{\E_M})=0\] since $[\tilde{b}_M]\in B(M, \tilde{w}_0\mu)$.

For the other side, suppose $x\in  P(C)w_0P_{\mu}(C)/P_{\mu}(C)$ with $\mathrm{pr}_{\tilde w_0}(x)\in\Fc(M, \tilde w_0\mu, \tilde{b}_M)^{wa}(C)$. We want to show $x\in \Fc(G, \mu, b)^{wa}(C)$. Suppose $Q$ is a standard parabolic subgroup of $G$ with Levi-component $M_Q$. Let $b_{M_Q}$ be a reduction of $b$ to $M_Q$. We need to show
\begin{eqnarray}\label{eqn_wa2}\deg\chi_*(\E_{b,x})_Q\leq 0, \forall \chi\in X^*(Q/Z_G)^+,\end{eqnarray} where $(\E_{b,x})_Q$ is the reduction to $Q$ of $\E_{b, x}$ induced by the reduction $\E_{b_{M_Q}}\times^{M_Q}Q$ to $Q$ of $\E_b$.
By the proof of Corollary \ref{Coro_HN equality among reduction to M}, the reduction $(\E_{b,x})_Q$ to $Q$ of $\E_{b,x}$ induces a reduction \[(\E_{\tilde{b}_M, \mathrm{pr}_{w_0}(x)}\times^M G)_{Q}\] of $\E_{\tilde{b}_M, \mathrm{pr}_{w_0}(x)}\times^M G$ to $Q$.
Moreover, we have
\[\deg \chi_*(\E_{b,x})_Q=\deg \chi_*(\E_{\tilde{b}_M, \mathrm{pr}_{w_0}(x)}\times^M G)_Q.\] Hence the inequality (\ref{eqn_wa2}) follows by the weak admissibility of $\mathrm{pr}_{w_0}(x)$ and Lemma \ref{lemma_wa Levi in wa G}.
\end{proof}

\section{The action of $\tilde{J}_b$ on the modifications of $G$-bundles}\label{section_group action on the modification}

Now we state the main results of this article. We first state the main result in the quasi-split case.

Suppose $G$ is quasi-split and $[b]\in A(G, \mu)$. Recall that by Remark \ref{AGmu}, we may assume $[b]\in B(G, \mu^{\sharp}+\e, \mu^{\diamond})$ for some $\e\in \pi_1(G)_{\Gamma, tor}$. By Lemma \ref{lemma_HN-decomp} (cf. also Corollary \ref{coro-HN-decomp}), there exists a unique standard Levi subgroup $M$ of $G$ with a reduction $b_M$ of $b$ to $M$ such that $(M, \mu, b)$ is HN-indecomposable and $\e$ is in the image of the injective map $\pi_1(M)_{\Gamma, tor}\rightarrow \pi_1(G)_{\Gamma, tor}$. So we may consider $\e$ as an element in $\pi_1(M)_{\Gamma, tor}$. In the following, sometimes we also identify $b_M$ with $b$.

\begin{theorem}\label{theo_main} Suppose $G$ is quasi-split and $\mu$ is minuscule. Let $[b]\in B(G, \mu^{\sharp}+\e, \mu^{\diamond})$. Suppose $M$ is the standard Levi subgroup of $G$ such that $(M, \mu, b)$ is Hodge-Newton-indecomposable. Then the equality $\Fc(G,\mu, b)^{wa}=\Fc(G,\mu, b)^{a}$ holds if and only if  $B(M,\mu^{\sharp}+\e, \mu^{\diamond})$ is fully Hodge-Newton-decomposable and $[b]$ is basic in $B(M)$.

In particular, for $[b]\in B(G, \mu)$, the equality $\Fc(G,\mu, b)^{wa}=\Fc(G,\mu, b)^{a}$ holds if and only if the pair $(M, \mu)$ is fully Hodge-Newton-decomposable and $[b]$ is basic in $B(M)$.
\end{theorem}
\begin{remark}For $\mathrm{GL}_n$ and $[b]\in B(G, \mu)$, this theorem is proved by Hartl \cite{Har} Theorem 9.3.\end{remark}

We have a similar description when the group $G$ is non-quasi-split .

\begin{theorem}\label{theo_main_general}Suppose $\mu$ is minuscule and $[b]\in B(G, \mu)$. Let $[b^H]\in B(H)$ be the image of $[b]$ via the natural morphism $B(G)\simeq B(H)$. Let $M^H$ be the standard Levi subgroup of $H$ with a reduction $b^H_M$ of $b^H$ to $M^H$ such that  $[b^H_M]\in A(M^H, \mu)$ and $(M^H, \mu, b^H_M)$ is HN-indecomposable. Then \begin{enumerate}\item the Levi-subgroup $M^H$ of $H$ corresponds to a Levi-subgroup $M$ of $G$ via the inner twisting,
\item the equality $\Fc(G,\mu, b)^{wa}=\Fc(G,\mu, b)^{a}$ holds if and only if the pair  $(M, \mu)$ is fully HN-indecomposable and $[b_M]$ is basic in $B(M)$, where $[b_M]$ corresponds to $[b_M^H]$ via $B(M)\simeq B(M^H)$.
\end{enumerate}
\end{theorem}

In order to prove the main theorems, we need some preparations.

\begin{lemma}\label{lemma_levi_oppo}Suppose $M$ is a standard Levi subgroup of $H$ defined over $F$.  Let $\dot w_0\in G(\breve F)$ be a representative of $w_0$. Then the map
\[\begin{split} M(\breve F)&\simeq (w_0Mw_0^{-1})(\breve F)\\ g &\mapsto \dot w_0 g\sigma(\dot w_0)^{-1} \end{split}\] induces the bijections \[\begin{split}B(M)&\simeq B(w_0 Mw_0^{-1}),\\ B(M)_{basic}&\simeq B(w_0 Mw_0^{-1})_{basic},\\ B(M, \mu)&\simeq B(w_0Mw_0^{-1}, \tilde{w}_0\mu),\\ B(M, \mu^{\sharp}+\e, \mu^{\diamond})&\simeq B(w_0Mw_0^{-1}, (\tilde{w}_0\mu)^\sharp+w_0\e, (\tilde{w}_0\mu)^{\diamond}),\end{split}\] where $\tilde{w}_0\mu:=(w_0\mu)_{w_0Mw_0^{-1}-dom}$ is the $w_0Mw_0^{-1}$-domominant representative of $w_0\mu$ in its $W_{w_0Mw_0^{-1}}$-orbit. Moreover $B(M, \mu^{\sharp}+\e, \mu^{\diamond})$ is fully HN-decomposable if and only if so is \\ $B(w_0Mw_0^{-1}, (\tilde{w}_0\mu)^\sharp+w_0\e, (\tilde{w}_0\mu)^{\diamond})$. \end{lemma}
\begin{proof}We may assume that $\tilde{w}_0$ is the minimal length representative in $W_{w_0Mw_0^{-1}}w_0$. Note that for any $g\in M(\breve F)$,
\[ [\nu_{\dot w_0g\sigma(\dot w_0)^{-1}}]= \tilde{w}_0[\nu_g]\in\N(w_0Mw_0^{-1}).\] The other assertions can be checked easily.\end{proof}

We also need the following proposition which is a key ingredient of the proof of the main result.

\begin{proposition}\label{Prop_not basic}Suppose $G$ quasi-split. Suppose $[b]\in A(G, \mu)$ and $(G,\mu, b)$ is HN-indecomposable. If $b$ is not basic, then $\Fc(G, \mu, b)^a\neq \Fc(G, \mu, b)^{wa}$.
\end{proposition}

Before proving this proposition, we first show how to use it combined with Proposition \ref{Prop_preimage_admissible} to prove the main theorem.

\begin{proof}[Proof of theorem \ref{theo_main}]For the sufficiency, suppose $B(M, \mu^{\sharp}+\e, \mu^{\diamond})$ is fully HN-decomposable and $[b_M]$ is basic in $B(M)$. Then $B(\tilde{M}, (\tilde{w}_0\mu)^\sharp+w_0\e, (\tilde{w}_0\mu)^{\diamond})$ is fully HN-decomposable and $[\tilde{b}_M]$ is basic in $B(\tilde{M})$ with $\tilde{M}=w_0Mw_0^{-1}$ and $\tilde{b}_M=w_0b_M\sigma(w_0)^{-1}$ by Lemma \ref{lemma_levi_oppo}. After applying \cite{CFS} Theorem 6.1 and its proof for the non quasi-split case to the triple $(\tilde{M},\tilde w_0\mu, \tilde{b}_M)$, we get that $\Fc(\tilde{M},\tilde w_0\mu, \tilde{b}_M)^a=\Fc(\tilde{M},\tilde w_0\mu, \tilde{b}_M)^{wa}$. The result follows by Proposition \ref{Prop_preimage_admissible}.

For the necessity, if $b_M$ is not basic in $M$ or $B(M, \mu^{\sharp}+\e, \mu^{\diamond})$ is not fully HN-indecomposable, it's equivalent to say that $\tilde{b}_M$ is not basic in $\tilde{M}$ or $B(\tilde{M}, (\tilde{w}_0\mu)^\sharp+w_0\e, (\tilde{w}_0\mu)^{\diamond})$ is not fully HN-indecomposable, then after applying Proposition \ref{Prop_not basic} (in the first case) or \cite{CFS} Theorem 6.1 and its proof for non quasi-split case to the triple $(\tilde{M},\tilde w_0\mu, \tilde{b}_M)$, we get $\Fc(M,\tilde w_0\mu, \tilde{b}_M)^a\neq\Fc(M,\tilde w_0\mu, \tilde{b}_M)^{wa}$. And $\Fc(G,\mu, b)^{wa}=\Fc(G,\mu, b)^{a}$ follows again by Proposition\ \ref{Prop_preimage_admissible}.
 \end{proof}

 \begin{proof}[Proof of theorem \ref{theo_main_general}] The reduction to the quasi-split case is similar to the proof of \cite{CFS} Theorem 6.1. We may assume that $G$ is adjoint, then $G=J_{b^*}$ is an extended pure inner form of $H$, where $[b^*]\in B(H)_{basic}$. We have \[\begin{split}\Fc(G, \mu, b)^a&=\Fc(H, \mu, b^H)^a,\\ \Fc(G, \mu, b)^{wa}&=\Fc(H, \mu, b^H)^{wa},\end{split} \]and $[b^H]\in B(H, \mu^\sharp+\kappa(b^*), \mu^{\diamond})$. By Lemma \ref{lemma_BGmu reduction Levi},
\[\kappa(b^*)\in \mathrm{Im}(\pi_1(M^H)_{\Gamma, tor}\rightarrow \pi_1(H)_{\Gamma, tor})\]
 and if we view $\kappa(b^*)$ as an element in $\pi_1(M^H)_{\Gamma, tor}$, we have
\[[b_M^H]\in B(M^H, \mu^{\sharp}+\kappa(b^*), \mu^{\diamond}).\] Let $[b^*_M]_{M^H}\in B(M^H)_{basic}$ with $\kappa_{M^H}(b^*_M)=\kappa(b^*)\in \pi_1(M^H)_{\Gamma}$. Then $[b^*]=[b^*_M]\in B(H)$ and we may assume $b^*=b^*_M\in M^H(\breve F)$. It follows that $M=J_{b^*_M}$ is a pure inner form of $M^H$ which is also a Levi subgroup of $G$. Therefore (1) follows.

For (2), by Theorem \ref{theo_main}, $\Fc(H, \mu, b^H)^a=\Fc(H, \mu, b^H)^{wa}$ if and only if $[b^H_M]$ is basic in $H$ and  $B(M^H, \mu^{\sharp}+\kappa(b^*), \mu^{\diamond})$ is fully HN-indecomposable. The later condition is equivalent to $(M, \mu)$ fully HN-indecomposability.
\end{proof}

The rest of the section is devoted to prove Proposition \ref{Prop_not basic}. Suppose $G$ is quasi-split. In order to distinguish the roots for different groups, we will write $\Delta_G$ and $\Delta_{G, 0}$ for $\Delta$ and $\Delta_0$ respectively. For any $\beta\in \Delta_G$, let $M_{\beta}$ be the stardard Levi subgroup of $G$ such that $\Delta_{M_{\beta}}=\Delta_G\backslash \Gamma\beta$. For any $\alpha\in \Delta_{G, 0}$, let $M_{\alpha}:=M_{\beta}$ for any $\beta\in \Delta_G$ such that $\beta|_A=\alpha$.


\begin{lemma}\label{Lemma_minimal_Mb-modification}Suppose $G$ is quasi-split. Let $[b]\in B(G, \mu^{\sharp}+\e, \mu^{\diamond})$ be a non basic element, where $\e\in \pi_1(G)_{\Gamma, tor}$. Suppose $(G, \mu, b)$ is HN-indecomposable.  Then there exist  $\tilde{\alpha}\in\Delta_0$, $w\in W$ and $x\in \Fc(M_{\tilde{\alpha}}, w\mu)(C)$ satisfying the following conditions:
\begin{enumerate}
\item $\langle\tilde \alpha, [\nu_b]\rangle >0$,
\item $w\mu$ is $M_{\tilde{\alpha}}$-dominant,
\item $\e\in \mathrm{Im}(\pi_1(M_{\tilde{\alpha}})_{\Gamma, tor}\hookrightarrow \pi_1(G)_{\Gamma, tor}$, (hence we may view $\e\in \pi_1(M_{\tilde{\alpha}})_{\Gamma, tor}$)
\item $\E_{b_{M_{\tilde\alpha}}, x}\simeq \E_{b'_{M_{\tilde\alpha}}}$, where $b_{M_{\tilde\alpha}}$ is the reduction of $b$ to $M_{\tilde\alpha}$ deduced from its canonical reduction $b_{M_b}$ to $M_b$ combined with the inclusion $M_b\subseteq M_{\tilde\alpha}$, and $b'_{M_{\tilde\alpha}}\in B(M_{\tilde{\alpha}})_{basic}$ such that $\kappa_{M_{\tilde{\alpha}}}(b'_{M_{\tilde\alpha}})=-(\tilde\beta^\vee)^\sharp+\e$ with $\tilde\beta\in \Delta$ and $\tilde\beta_{|A}=\tilde{\alpha}$.
\end{enumerate}
\end{lemma}

\begin{proof} By Lemma \ref{lemma_BGmu reduction Levi} $[b_{M_b}]\in B(M_b, (w_1\mu)^{\sharp}+\e, (w_1\mu)^{\diamond})$ for some $w_1\in W$. Moreover, $\e\in \mathrm{Im}(\pi_1(M_b)_{\Gamma, tor}\hookrightarrow \pi_1(G)_{\Gamma, tor}$. Hence (3) holds for any $M_{\tilde\alpha}$ containing $M_b$.  For any $\alpha\in \Delta_{G, 0}$ such that $\langle\alpha, [\nu_b]\rangle >0$, let $[b_{M_\alpha}]$ be the image of $[b_{M_b}]$ via the natural map $B(M_b)\rightarrow B(M_\alpha)$. Then $[b_{M_{\alpha}}]\in A(M_\alpha, (w_1\mu)_{M_\alpha-dom})$.  As $(b, \mu)$ is not HN-decomposable, we have $(w_1\mu)_{M_{\alpha}-dom}\neq \mu$. Therefore there exists $\beta\in \Delta_G$ such that $\beta|_A=\alpha$ and $\langle\beta, (w_1\mu)_{M_\alpha-dom}\rangle<0$.  Let
\[R_{\alpha}:=\{(s_{\beta}(w_1\mu)_{M_{\alpha}-dom})_{M_{\alpha}-dom}\in W\mu| \beta\in \Delta_G, \beta|_A=\alpha, \langle\beta, (w_1\mu)_{M_\alpha-dom}\rangle<0\},\] where $s_\beta\in W$ is the reflection corresponding to $\beta$.
Let $w\mu$ be a maximal element in the subset \[\bigcup_{\alpha\in \Delta_{G, 0}, \langle\alpha, [\nu_b]\rangle > 0} R_\alpha\subset W\mu.\] Suppose $w\mu\in R_{\alpha}$ for any $\alpha\in \Delta_{G, 0}$, then $(\alpha, w\mu)$ satisfies (1) and (2). It remains to find $x\in \mathcal{F}(M_{\tilde\alpha}, w\mu)(C)$ satisfying condition (4) for some $\tilde\alpha$ with $w\mu\in R_{\tilde\alpha}$.

It's equivalent to find $\tilde\alpha\in\Delta_{G,0}$, $y\in \Fc(M_{\tilde\alpha}, (w\mu)^{-1})(C)$ such that $\E_{b_{M_{\tilde\alpha}}}\simeq \E_{b'_{M_{\tilde\alpha}}, y}$ and $w\mu\in R_{\tilde\alpha}$.

By Proposition \ref{prop:description modification basique}, we need to find $\tilde\alpha\in \Delta_{G,0}$ such that $w\mu\in R_{\tilde\alpha}$ and
\[ [b_{M_{\tilde\alpha}}]\in B(M_{\tilde\alpha}, \kappa_{M_{\tilde\alpha}} (b'_{M_{\tilde\alpha}})+(w\mu)^\sharp, \nu_{b'_{M_{\tilde\alpha}}}(w\mu)^{\diamond}).\] By Lemma \ref{lemma_BGmu}, the later is equivalent to the conditions
\begin{eqnarray}\label{eqn_kappa}\kappa_{M_{\tilde\alpha}}(b_{M_{\tilde\alpha}})&=&(w\mu-\tilde\beta^{\vee})^{\sharp}+\e \text{ in } \pi_1(M_{\tilde\alpha})_{\Gamma}, \\
\label{eqn_nu} \kappa_{M_b}(b_{M_b})&\preceq_{M_b}& -\mathrm{Av}_{\Gamma}\mathrm{Av}_{W_{M_{\tilde\alpha}}}(\tilde\beta^{\vee})+\mathrm{Av}_{\Gamma}(w\mu) \text{ in } \pi_1(M_b)_{\Gamma, \Q}.\end{eqnarray}

Let $(\beta_j)_{j\in J}$ be a set of representatives of Galois orbits in $\Delta_G\backslash \Delta_{M_b}$. Let
\[w\mu-w_1\mu=\sum_{j\in J}n_j\beta^{\vee}_j \text{ in }\pi_1(M_b)_{\Gamma}\] with $n_j\in\mathbb{N}$ for all $j\in J$.
\

\noindent {\it Claim 1: $n_j\geq 1, \forall j\in J$.}

Suppose $n_{j_0}=0$ for some $j_0\in J$. Let $\alpha_0=\beta_{j_0}|_A$. Then
$(w\mu)_{M_{\alpha_0}-dom}=(w_1\mu)_{M_{\alpha_0}-dom}$. Again by HN-indecomposability, $(w\mu)_{M_{\alpha_0}-dom}\neq \mu$.
Then there exists $\beta\in \Gamma\beta_{j_0}$ such that $\langle\beta, (w\mu)_{M_{\alpha_0}-dom} \rangle<0$. It follows that
\[(s_{\beta}(w_1\mu)_{M_{\alpha_0}-dom})_{M_{\alpha_0}-dom}\succneqq w\mu.\]This contradicts with the maximality of $w\mu$. Hence Claim 1 follows.

By definition, suppose $w\mu\in R_{\alpha}$ for some $\alpha$ and $w\mu=(s_{\beta}(w_1\mu)_{M_{\alpha}-dom})_{M_{\alpha}-dom}$, where $\alpha\in \Delta_{G,0}$ with $\langle\alpha, [\nu_b]\rangle>0$ and $\beta\in \Delta_G$ with $\beta|_A=\alpha$. Suppose $\beta\in \Gamma\beta_{j_0}$ for some $j_0\in J$, then $n_{j_0}=1$ by the definition of $w\mu$. If the subset
\[J_1:=\{j\in J| n_j=1\}\] of $J$ consists of the single element of $j_0$. Let $\tilde{\alpha}:=\alpha$, $\tilde{\beta}:=\beta$. We will verify that $\tilde\alpha$ satisfies the desired properties. Since $[b_{M_b}]\in B(M_b, (w_1\mu)^{\sharp}+\e, (w_1\mu)^{\diamond})$, we have
\[\kappa_{M_{\tilde{\alpha}}}(b_{M_{\tilde\alpha}})= (w_1\mu)^{\sharp}+\e= (w\mu-\tilde{\beta}^{\vee})^{\sharp}+\e \text{ in } \pi_1(M_{\tilde\alpha})_{\Gamma}.\] The equality (\ref{eqn_kappa}) follows. For the inequality (\ref{eqn_nu}), we have
\begin{eqnarray*}
& &\mathrm{Av}_{\Gamma}(w\mu)-\kappa_{M_b}(b_{M_b}) -\mathrm{Av}_{\Gamma}\mathrm{Av}_{W_{M_{\tilde\alpha}}}(\tilde\beta^{\vee})\\
&=& (w\mu-w_1\mu)-\mathrm{Av}_{W_{M_{\tilde\alpha}}}(\tilde\beta^{\vee})\\
&\succcurlyeq_{M_b} & 0
\end{eqnarray*}
in $\pi_1(M_b)_{\Gamma,\Q}$, where the last inequality follows from Lemma \ref{lemma_average coroot} (1) (because $E_8$ case does not occur as there is only trivial minuscule cocharacters in that case) combined with the fact that $n_j\geq 2$ for any $j\in J\backslash\{j_0\}$.

It remains to deal with the case when $J_1$ has at least two elements. By Claim 1, for any $j\in J$, up to replacing $\beta_j$ by some other representative in the same Galois orbit, we may assume $\beta_j$ appears in the linear combination of $w\mu-w_1\mu$.
\

\noindent {\it Claim 2: $w\mu=\mu=(s_{\beta_{j'}}(w_1\mu)_{M_{\beta_{j'}}-dom})_{M_{\beta_{j'}}-dom}$ for any $j'\in J_1$.}

We want to show $w\mu$ is $G$-dominant. Suppose $j'_0\in J_1$ with $j_0\neq j'_0$. Let $w\mu=w_3s_{\beta_{j_0'}}w_2w_1\mu$ with $w_2, w_3\in W_{M_{\beta_{j'_0}}}$. Then we have
\[(s_{\beta_{j_0'}}(w_1\mu)_{M_{\beta_{j'_0}}-dom})_{M_{\beta_{j'_0}}-dom}\succcurlyeq s_{\beta_{j'_0}}w_2 w_1\mu.\] Since both sides are in $W\mu$ with difference a linear combination of coroots in $M_{\beta_{j'_0}}$, we have
\[\bigcup_{\alpha}R_{\alpha}\ni (s_{\beta_{j_0'}}(w_1\mu)_{M_{\beta_{j'_0}}-dom})_{M_{\beta_{j'_0}}-dom}=(s_{\beta_{j'_0}}w_2 w_1\mu)_{M_{\beta_{j'_0}}-dom}\succcurlyeq w\mu.\] By the maximality of $w\mu$, we deduce that $w\mu$ is both $M_{\beta_{j_0}}$-dominant and $M_{\beta_{j_0'}}$-dominant. Therefore \[w\mu=(s_{\beta_{j'_0}}w_2 w_1\mu)_{M_{\beta_{j'_0}}-dom}\] is $G$-dominant and Claim 2 follows.

Let $(I_i)_{0\leq i\leq r}$ be the increasing sequence of subsets in $\Delta_G$ as in Lemma \ref{lemma_average coroot}. Suppose $i_0$ is the smallest integer such that $\{\beta_{j}| j\in J_1\}\cap I_{i_0}$ is not an empty set. Choose $\tilde\beta\in \{\beta_{j}| j\in J_1\}\cap I_{i_0}$. Let $\tilde\alpha:=\tilde\beta|_A$. By the same arguments as before we can verify that $\tilde\alpha$ satisfies the condition (\ref{eqn_kappa}). For the condition (\ref{eqn_nu}), let
\[\mathrm{Av}_{W_{M_{\tilde\alpha}}}(\tilde\beta^{\vee}):=\sum_{j\in J}m_j \beta_j^{\vee}\text{ in } \pi_1(M_b)_{\Gamma,\mathbb{Q}},\]with $m_j\in\mathbb{Q}$ for all $j\in J$. By Lemma \ref{lemma_average coroot},  $0\leq m_j\leq 2$ (as $E_8$ case will not occur) for all $j\in J$, and $0\leq m_j\leq 1$ for all $j\in J_1$. Then
\begin{eqnarray*}
& &\mathrm{Av}_{\Gamma}(w\mu)-\kappa_{M_b}(b_{M_b})- \mathrm{Av}_{\Gamma}\mathrm{Av}_{W_{M_{\tilde\alpha}}}(\tilde\beta^{\vee})\\
&=& (w\mu-w_1\mu)-\mathrm{Av}_{W_{M_{\tilde\alpha}}}(\tilde\beta^{\vee})\\
&=&\sum_{j\in J}(n_j-m_j)\beta_j^{\vee}\\
&\succcurlyeq_{M_b} &\sum_{j\in J_1}(1-m_j)\beta_j^{\vee}+ \sum_{j\in J\backslash J_1}(2-m_j)\beta_j^{\vee}\\
&\succcurlyeq_{M_b} & 0.
\end{eqnarray*}
\end{proof}

\begin{lemma}\label{lemma_average coroot}Suppose $G$ is quasi-split. Let $\beta\in \Delta_G$. Suppose
\begin{eqnarray}\label{eqn_average}\mathrm{Av}_{M_{\beta}}\beta^{\vee}=\beta^{\vee}+\sum_{\gamma\in (\Delta_{M_{\beta}})_{\Gamma}}n_{\beta, \gamma}\gamma^{\vee} \text{ in }X_*(T)_{\Gamma,\Q}.\end{eqnarray}Then
\begin{enumerate}
\item $0\leq n_{\beta, \gamma}\leq 3$ for all $\gamma\in(\Delta_G)_{\Gamma}$. Moreover, if all the connected components of the Dynkin diagram of $G$ is not of type $E_8$, then
$0\leq n_{\beta, \gamma}\leq 2$ for all $\gamma\in(\Delta_G)_{\Gamma}$.
\item There exist $r\in\mathbb{N}_{\geq 1}$ and an increasing sequence of $\Gamma$-invariant subsets in $\Delta_G$
\[\emptyset=I_0\subset I_1\subset\cdots\subset I_{r-1}\subset I_r=\Delta_G \] such that if $\beta\in I_i$ for some $i$, then $n_{\beta,\gamma}\leq 1$ for all $\gamma\notin (I_{i-1})_{\Gamma}$. In particular, if $r=1$, then $n_{\beta, \gamma}\leq 1$ for all $\beta\in \Delta_{G}$, $\gamma\in (\Delta_G)_{\Gamma}$.
\end{enumerate}
\end{lemma}
\begin{proof} This lemma only depends on the absolute root system of $G$ with Galois action. After considering separately each connected component of the Dynkin diagram of $G$, we may assume the Dynkin diagram of $G$ is connected. The first assertion can be checked directly case by case. Indeed, it suffices to compute explicitely all the $n_{\beta, \gamma}$ in (\ref{eqn_average}). As $\langle\mathrm{Av}_{M_{\beta}}\beta^{\vee}, \alpha\rangle=0$ for any $\alpha\in \Delta_{M_{\beta}}$, it follows that
\[\langle\beta^{\vee}, \alpha\rangle+\sum_{\gamma\in (\Delta_{M_{\beta}})_{\Gamma}}n_{\beta, \gamma}\langle\gamma^{\vee}, \alpha\rangle=0\]
for any $\alpha\in \Delta_{M_{\beta}}$. Then $\{n_{\beta, \gamma}\}$ is the unique solution of this system of linear equations.

 For the second assertion, we list the increasing sequence of $\Gamma$-invariant subsets in $\Delta_G$ case by case according to the type of Dynkin diagram of $G$. We can check directly that this increasing sequence of subsets satisfies the desired property. We left the details of the verification to the readers.

\noindent{\it Case $A_n$:} In the $^1A_n$ case or $^2A_n$ case with $n$ even, take $r=1$. Otherwise, we are in the $^2A_n$ case with $n$ odd, then take $r=2$ and $I_1=\Delta_G\backslash\{\beta\}$ where $\beta$ is the unique $\Gamma$-invariant root in $\Delta_G$.

\noindent{\it Case $B_n$:} take $r=2$ and $I_1$ is the subset of long roots in $\Delta_G$.

\noindent{\it Case $C_n$:} take $r=1$.

\noindent{\it Case $D_n$:} In the $^1D_n$ case with $n>4$, take $r=2$ and $\Delta_G\backslash I_1$ consists of two roots which are the end points of the Dynkin diagram and are neighbors to the same simple root.

In the $^2D_n$ case or $^1D_4$ case, take $r=1$.

In the $^3 D_4$ case, take $r=2$ and $I_1=\Delta_G\backslash \{\beta\}$ where $\beta$ is the unique $\Gamma$-invariant root in $\Delta_G$.

\noindent{\it Case $E_n$:} Suppose the Dynkin diagram of $E_8$ is as follows:
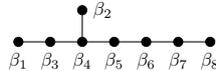
\[\begin{tikzpicture}
\dynkin[scale=1.2, text/.style={scale=1.8}]{E}{8};
\dynkinLabelRoot{1}{\beta_1}
\dynkinLabelRoot{2}{\beta_2}
\dynkinLabelRoot{3}{\beta_3}
\dynkinLabelRoot{4}{\beta_4}
\dynkinLabelRoot{5}{\beta_5}
\dynkinLabelRoot{6}{\beta_6}
\dynkinLabelRoot{7}{\beta_7}
\dynkinLabelRoot{8}{\beta_8}.
\end{tikzpicture}\]
In the $E_7$ (resp. $E_6$) case, we remove $\beta_8$ (resp. $\beta_7$ and $\beta_8$) in the diagram.

In the $^1E_6$ case, take $r=3$, $I_1=\{\beta_4\}$, $I_2=\{\beta_2, \beta_3, \beta_4, \beta_5\}$.

In the $^2E_6$ case, take $r=4$, $I_1=\{\beta_3, \beta_5\}$, $I_2=I_1\cup\{\beta_1,\beta_6\}$, $I_3=I_2\cup\{\beta_4\}$.

In the $E_7$ case, take $r=4$, $I_1=\{\beta_4\}$, $I_2=\{\beta_3, \beta_4, \beta_5\}$, $I_3=\Delta_G\backslash\{\beta_7\}$.

In the $E_8$ case, take $r=5$,  $I_1=\{\beta_4\}$, $I_2=\{\beta_4, \beta_5\}$, $I_3:=I_2\cup\{\beta_3, \beta_6\}$, $I_4=\Delta_G\backslash\{\beta_8\}$.

\noindent{\it Case $F_4$:} Suppose the Dynkin diagram is as follows:
\[\begin{tikzpicture}
\dynkin[scale=1.2, text/.style={scale=1.8}]{F}{4};
\dynkinLabelRoot{1}{\beta_1}
\dynkinLabelRoot{2}{\beta_2}
\dynkinLabelRoot{3}{\beta_3}
\dynkinLabelRoot{4}{\beta_4}.
\end{tikzpicture}\]
Take $r=3$, $I_1=\{\beta_2\}$, $I_2=\{\beta_1, \beta_2, \beta_3\}$.

\noindent{\it Case $G_2$:} Take $r=2$, $I_1$ is the set of unique long root in $\Delta_G$.
\end{proof}

Now we can prove Proposition \ref{Prop_not basic}.
\begin{proof}[Proof of Proposition \ref{Prop_not basic}]
As $b$ is not basic, there exist $\alpha\in\Delta_0$, $w\in W$ and $x\in \Fc(M_{\tilde \alpha}, w\mu)(C)$ satisfies the properties (1)-(3) in Lemma \ref{Lemma_minimal_Mb-modification} which follows. Let $b_{M_{\tilde \alpha}}$ and $b'_{M_{\tilde \alpha}}$ as in loc. cit.

Let $x_G\in \Fc(G, \mu)(C)$ be the image of $x$ via the natural morphism:
\[\begin{array}{ccc}\Fc(M_{\tilde\alpha}, w\mu)(C)&\rightarrow& \Fc(G, \mu)(C)\\ a(M_{\tilde\alpha}\cap wP_{\mu}w^{-1}) &\mapsto& aw P_{\mu}\end{array}\]Then $\E_{b, x_G}\simeq \E_{b'_{M_{\tilde\alpha}}}\times^{M_{\tilde\alpha}} G\simeq\E_{b'}$ where $[b']\in B(G)$ is the image of $[b'_{M_{\tilde\alpha}}]\in B(M_{\tilde\alpha})$ via the natural map $B(M_{\tilde\alpha})\rightarrow B(G)$. Hence $x_G\notin \Fc^a(C)$. Moreover, $x_G\notin \Fc^{wa}(C)$. Indeed, the canonical reduction $(\E_{b'})_{P_{\tilde\alpha}}$ of $\E_{b'}$ to the standard parabolic subgroup $P_{\tilde\alpha}$ corresponding to $M_{\tilde\alpha}$ induces the reduction $\E_{b_{M_{\tilde\alpha}}}\times^{M_{\tilde\alpha}}P_{\tilde\alpha}$ of $\E_b$ to $P_{\tilde\alpha}$. Take $\chi\in X^*(P_{\tilde\alpha}/Z_G)^{+}$,  then $\mathrm{deg}\chi_*(\E_{b'})_{P_{\tilde\alpha}}>0$.

For any element $\gamma\in\tilde J_b (K)$, we have $\gamma(x_G)\notin\Fc^a(C)$. Choose $\gamma\in\tilde{J}^{\geq \lambda_{max}}_b(K)\backslash\{1\}$ with  $\lambda_{max}=\max_{\gamma\in \Phi}\langle \nu_b, \gamma\rangle$, it remains to show that $\gamma(x_G)\in \Fc^{wa}(C)$.

Suppose the pair $(b,\gamma(x_G))$ is not weakly admissible. There exist a standard maximal parabolic subgroup $Q$, a reduction $b_{M_Q}$ of $b$ to the Levi component $M_Q$ of $Q$ and $\chi \in X^* (Q/Z_G)^+$ such that
\[\deg \, \chi_* (\E_{b,\gamma(x_G)})_Q >0,\]
 where $(\E_{b,\gamma(x_G)})_Q$ is the reduction of $\E_{b,\gamma(x_G)}$ to $Q$ induced by a reduction $\E^{\gamma}_{b, Q}$ of $\E_b$ to $Q$, where $\E^{\gamma}_{b, Q}:=\E_{\tilde b_{M_Q}}\times^{M_Q} Q$ is induced by a reduction $\tilde b_{M_Q}$ of $b$ to $M_Q$ (and hence to $Q$).
The isomorphism $\gamma: \E_{b}\stackrel{\sim}{\rightarrow}\E_{b}$ induces an isomorphism $\E_{b}/Q\stackrel{\sim}{\rightarrow}\E_{b}/Q$, hence by Remark \ref{remark_section}, $\E_{b, Q}$ induces a reduction $\E_{b, Q}^{\gamma}$ of $\E_b$ to $Q$ and an isomorphism $\E_{b, Q}\stackrel{\sim}{\rightarrow} \E_{b, Q}^{\gamma}$ satisfying the commutative diagram
\[\xymatrix{\E_b\ar[r]^{\gamma}_{\sim}\ar[d]^{\sim} &\E_b\ar[d]^{\sim}\\ \E_{b, Q}\times^Q G\ar[r]^{\sim} &\E^{\gamma}_{b, Q}\times^Q G}  \]

Suppose the reduction $(\E_{b, x_G})_Q$ to $Q$ of $\E_{b, x_G}$ is induced by the reduction $(\E_{b, \gamma(x_G)})_Q$ to $Q$ of $\E_{b, \gamma(x_G)}$ by Lemma \ref{lemma_Jbaction} and Remark \ref{remark_section}. We get a cubic commutative diagram

\[\xymatrix{ &\E_{b, Q}\ar[rr]^{\sim}\ar[dl]\ar@{~~>}[dd] &&\E_{b, Q}^{\gamma}\ar@{~>}[dd]\ar[dl]\\
 \E_b\ar[rr]^(0.6){\sim}_(0.6){\gamma}\ar@{~>}[dd]  &&\E_b\ar@{~>}[dd] & \\
 &(\E_{b, x_G})_Q\ar@{-->}[dl]\ar@{-->}[rr]^(0.4){\sim}  &&(\E_{b, \gamma(x_G)})_Q\ar[dl]  \\
 \E_{b, x_G}\ar[rr]^{\sim } &&\E_{b, \gamma(x_G)} & }  \]
where the vertices of the front face are the $G$-bundles and vertices of the back face are the corresponding reductions of $G$-bundles to $Q$,  the vertical waved arrows denote the modification of $G$ or $Q$-bundles. It follows that $\deg\chi_*(\E_{b, x_G})_Q=\deg\chi_*(\E_{b, \gamma(x_G)})_Q>0$.

According to theorem \ref{theoSchi} the vector
\begin{align*}
v: X^* (Q/Z_G)&\longrightarrow \Z \\
\chi &\longmapsto \deg\, \chi_*  (\E_{b,x_G})_Q
\end{align*}
seen as an element of $X_*(A)_{\Q}$ satisfies $v\leq \nu_{\E_{b,x_G}}=\nu_{\E_{b'}}$ with $v\neq 0$. As $\nu_{\E_{b'}}\in\mathcal{N}_G\backslash\{0\}$ is minimal, one deduces that $v=\nu_{b, x_G}$,  $Q=P_{\tilde\alpha}$ and  $(\E_{b,x_G})_Q$ is the HN canonical reduction of $\E_{b,x_G}$. Therefore $\E_{b,Q}=\E_{b_{M_{\tilde\alpha}}}\times^{M_{\tilde\alpha}}P_{\tilde\alpha}$ to $Q$ of $\E_b$. Pushing forward via the natural projection $P_{\tilde\alpha}\rightarrow M_{\tilde\alpha}$,  the isomorphism of $P_{\tilde\alpha}$-bundles $\E_{b, P_{\tilde\alpha}}\simeq \E_{b, P_{\tilde\alpha}}^{\gamma}$ induces an isomorphism of $M_{\tilde\alpha}$-bundles, hence we have $[b_{M_{\tilde\alpha}}]=[\tilde{b}_{M_Q}]\in B(M_{\tilde\alpha})$. According to the following lemma \ref{lemma_equivalent reduction}, the two reductions $b_{M_{\tilde\alpha}}$ and $\tilde{b}_{M_{\tilde\alpha}}:=\tilde{b}_{M_Q}$  of $b$ to $M_Q$ are equivalent. In particular, the reductions $\E_{b, Q}$ and $\E_{b, Q}^{\gamma}$ of $\E_b$ to $Q$ are equivalent. Hence they give the same sub-vector bundle
\[\E_{b, Q}\times^{Q, \mathrm{Ad}}\mathrm{Lie} Q=\E_{b, Q}^{\gamma}\times^{Q, \mathrm{Ad}}\mathrm{Lie} Q\]
 over $X$ of $\mathrm{Ad}(\E_b)=\E_b\times^{G, \mathrm{Ad}}\mathfrak{g}$. By theorem \ref{theorem_classification_barB-modules}, this sub-vector bundle corresponds to a sub-$\bar B$-module $\mathfrak{q}$ in $\mathfrak{g}_{\bar B}$ which is stable under the action of $\mathrm{Ad}(\gamma)$, where we identify $\tilde{J}_b(K)$ with a subgroup of $G(\bar B)$ (cf. section \ref{subsection_Jb}). Recall that in loc. cit., we have also defined a filtration $(\mathfrak{g}_{\bar B}^{\geq \lambda})_{\lambda\in\Q}$ on $\mathfrak{g}_{\bar B}$, As $\nu_{b_{M_{\tilde\alpha}}}=\nu_b$, the non-zero elements of $\mathfrak{q}$ are in the $-\langle\gamma, \nu_b\rangle$ graded pieces of $\mathfrak{g}_{\bar B}$ for some absolute root $\gamma$ in $Q=P_{\tilde\alpha}$. In particular $\mathfrak{q}\cap\mathfrak{g}^{\geq \lambda_{max}}_{\bar B}=0$. On the other side, since $\gamma\neq 1$, we can always choose an element $y\in\mathfrak{q}\cap\mathfrak{g}_{\bar B}^{\geq 0}$ such that $\mathrm{Ad}(\gamma)(y)\neq y$. Note that \[0\neq \mathrm{Ad}(\gamma)(y)-y\in \mathfrak{g}_{\bar B}^{\geq \lambda_{max}},\] which implies $\mathfrak{q}\cap\mathfrak{g}_{\bar B}^{\geq \lambda_{max}}\neq 0$.  We get a contradiction.
\end{proof}

\begin{lemma}\label{lemma_equivalent reduction} Suppose $b\in G(\breve F)$. Suppose $\tilde\alpha\in \Delta_0$ such that $\langle\tilde\alpha, \nu_b\rangle>0$. Let $(b_{M_{\tilde\alpha}}, g)$ and $(\tilde{b}_{M_{\tilde\alpha}}, \tilde{g})$ be two reductions of $b$ to $M_{\tilde\alpha}$ with $[b_{M_{\tilde\alpha}}]=[\tilde{b}_{M_{\tilde\alpha}}]\in B(M_{\tilde\alpha})$ and $\nu_{b_{M_{\tilde\alpha}}}=\nu_b$. Then the two reductions $(b_{M_{\tilde\alpha}}, g)$ and $(\tilde{b}_{M_{\tilde\alpha}}, \tilde{g})$ are equivalent.
\end{lemma}

\begin{proof}As $[b_{M_{\tilde\alpha}}]=[\tilde{b}_{M_{\tilde\alpha}}]\in B(M_{\tilde\alpha})$, we may assume $b_{M_{\tilde\alpha}}=\tilde{b}_{M_{\tilde\alpha}}$. Then
\[\tilde{g}g^{-1}\in J_b=\{h\in G(\breve F)| b\sigma (h)= hb \}.\]Since $\langle\tilde\alpha, \nu_b\rangle>0$, $J_b\subseteq M_{\tilde\alpha}(\breve F)$. It follows that $(b_{M_{\tilde\alpha}}, g)$ and $(\tilde{b}_{M_{\tilde\alpha}}, \tilde{g})$ are equivalent.
\end{proof}

\section{Newton stratification and weakly admissible locus}\label{section_wa and Newton}
In this section, we suppose $G$ is quasi-split and $[b]\in A(G, \mu)$ is basic. Under this condition, the proof of \cite{CFS} Theorem 6.1 in fact shows the following finer result.

\begin{theorem}[\cite{CFS}]\label{theo_CFS}Let $[b']\in B(G, \kappa_G(b)-\mu^{\sharp}, \nu_b\mu^{-1})$.
\begin{enumerate}
\item If $(G, \nu_b(w_0\mu^{-1})^{\diamond}, b')$ is HN-decomposable, then $\Fc(G, b, \mu)^{[b']}\cap \Fc(G, b, \mu)^{wa}=\emptyset$;
\item If $(G, \nu_b(w_0\mu^{-1})^{\diamond}, b')$ is HN-indecomposable and $[b']$ a minimal element in the set $B(G, \kappa_G(b)-\mu^{\sharp}, \nu_b\mu^{-1})\backslash [1]$ for the dominance order, then $\Fc(G, b, \mu)^{[b']}\cap \Fc(G, b, \mu)^{wa}\neq\emptyset$.
\end{enumerate}
\end{theorem}

Inspired by this theorem, we have the following conjecture.

\begin{conjecture}\label{conjecture}Suppose  $[b']\in B(G, \kappa_G(b)-\mu^{\sharp}, \nu_b\mu^{-1})$ with $(G, \nu_b(w_0\mu^{-1})^{\diamond}, b')$ HN-indecomposable, then \[\Fc(G, b, \mu)^{[b']}\cap \Fc(G, b, \mu)^{wa}\neq\emptyset.\]
\end{conjecture}

\begin{remark}The full conjecture is proved by Viehmann in \cite{Vi} very recently.
\end{remark}

In the rest of the section, we will prove this conjecture for the linear algebraic groups for special $\mu$.

For $r, s\in \Z$ with $r>0$,  let $\mathcal{O}([\frac{s}{r}]):=\mathcal{O} (\frac{s}{r})^{d}$ if $d=(s, r)$. Then $\mathrm{deg}\mathcal{O}([\frac{s}{r}])=s$ and $\mathrm{rank}\mathcal{O}([\frac{s}{r}])=r$.

\begin{proposition}\label{prop_conj special case}Let $G=\mathrm{GL}_n$. Suppose $[b']\in B(G, \kappa_G(b)-\mu^{\sharp}, \nu_b\mu^{-1})$ with $(G, \nu_b(w_0\mu^{-1})^{\diamond}, b')$ HN-indecomposable. If $\E_{b'}\simeq \mathcal{O}([\frac{2}{r_1}])\oplus \mathcal{O}([\frac{-2}{r_2}])$ or $\mathcal{O}([\frac{1}{r_1}])\oplus \mathcal{O}([\frac{0}{r_2}])\oplus \mathcal{O}([\frac{-1}{r_3}])$ for some $r_1, r_2, r_3>0$, then \[\Fc(G, b, \mu)^{[b']}\cap \Fc(G, b, \mu)^{wa}\neq\emptyset.\] In particular, the conjecture \ref{conjecture} holds when $\mu=(1^{(r)}, 0^{(n-r)})$ with $r(n-r)\leq 2n$.
\end{proposition}
\begin{proof} For the last assertion, if $\mu=(1^{(r)}, 0^{(n-r)})$ with $r(n-r)\leq 2n$, then by Proposition \ref{prop:description modification basique}, $[b']\in B(G, \kappa_G(b)-\mu^{\sharp}, \nu_b\mu^{-1})\backslash [1]$ implies $\E_{b'}\simeq \mathcal{O}([\frac{2}{r_1}])\oplus \mathcal{O}([\frac{-2}{r_2}])$ or $\mathcal{O}([\frac{1}{r_1}])\oplus \mathcal{O}([\frac{0}{r_2}])\oplus \mathcal{O}([\frac{-1}{r_3}])$ for some $r_1, r_3>0$ and $r_2\geq 0$. So the last assertion follows from the first one combined with Theorem \ref{theo_CFS} (when $r_2=0$). Now it remains to prove the first assertion.

  We claim that there exists an exact sequence of vector bundles
\begin{eqnarray}\label{eqn_extension}0\rightarrow \E'\rightarrow \E_b\rightarrow \E''\rightarrow 0\end{eqnarray}
 satisfying a commutative diagram
 \[
\xymatrix{0\ar[r]& \E' \ar[r] &\E_b \ar[r] &\E''\ar[r] &0\\
 0 \ar[r]& \tilde{\E}'\ar[r]\ar@{^(->}[u] &\E_{b'}\ar[r]\ar@{^(->}[u] &\tilde{\E}''\ar[r]\ar@{^(->}[u] &0  } \]
where $\E'$ and $\E''$ are semi-stable vector bundles and the vertical arrows are the modifications of minuscule type. Indeed, suppose $\E_b=\mathcal{O}([\frac{s}{r}])$.
If $\E_{b'}\simeq \mathcal{O}([\frac{2}{r_1}])\oplus \mathcal{O}([\frac{-2}{r_2}])$,
then $r=r_1+r_2$.  Let  $\tilde{\E}'=\mathcal{O}([\frac{2}{r_1}])$, $\tilde{\E}''=\mathcal{O}([\frac{-2}{r_2}])$. If $s\leq r_2$, then let $\E'=\mathcal{O}([\frac{2}{r_1}])$ and $\E''=\mathcal{O}([\frac{s-2}{r_2}])$. Otherwise, let $\E'=\mathcal{O}([\frac{s+2-r_2}{r_1}])$ and $\E''=\mathcal{O}([\frac{r_2-2}{r_2}])$.  If $\E'_b\simeq \mathcal{O}([\frac{1}{r_1}])\oplus \mathcal{O}([\frac{0}{r_2}])\oplus \mathcal{O}([\frac{-1}{r_3}])$, then $r_1+r_2+r_3=r$. We can easily check that one of the following two equalities holds:

\[\frac{s-1}{r-r_1}\leq \frac{r_3-1}{r_3},\ \   \frac{s-r_3+1}{r-r_3}\geq \frac{1}{r_1}.\](Otherwise, the equalities give upper and lower bounds for $s$. The comparison of the two bounds leads to a contradiction $r<r_1+r_3$). In the former case, let
\[\begin{array}{rlrl}
   \E'&=\mathcal{O}([\frac{1}{r_1}]), &\E''&=\mathcal{O}([\frac{s-1}{r-r_1}]), \\ \tilde{\E}'&=\mathcal{O}([\frac{1}{r_1}]), &\tilde{\E}''&=\mathcal{O}([\frac{0}{r_2}])\oplus\mathcal{O}([\frac{-1}{r_3}]) \end{array}\]
In the later case, let

\[\begin{array}{rlrl}
   \E'&=\mathcal{O}([\frac{s-r_3+1}{r-r_3}]), &\E''&=\mathcal{O}([\frac{r_3-1}{r_3}]), \\ \tilde{\E}'&=\mathcal{O}([\frac{1}{r_1}])\oplus\mathcal{O}([\frac{0}{r_2}]), &\tilde{\E}''&=\mathcal{O}([\frac{-1}{r_3}]) \end{array}\]

The existence of the extension \ref{eqn_extension} is due to \cite{BFH}, and the existence of the modifications given by the left and right vertical arrows are given by Proposition \ref{prop:description modification basique}.  The vertical arrow in the middle of the commutative diagram gives a point in $x\in\Fc(G, b, \mu)^{[b']}(C)$. It suffices to show $x\in\Fc^{wa}(C)$. Suppose $\mathcal{G}\subseteq \E_{b'}$ is any sub-vector bundle of $\E_{b'}$ corresponding to a reduction $(\E_{b, x})_P$ of $\E_{b'}$ to a maximal parabolic subgroup $P$ such that $\deg\mathcal{G}>0$. Let $\mathcal{G}'\subseteq \E_b$ be the corresponding sub-vector bundle of $\E_b$ corresponding to the reduction $(\E_b)_P$ induced by $(\E_{b, x})_P$. We want to show that $\mathcal{G}'$ does not come from subisocrystals.

Suppose $\mathcal{G}'$ comes from subisocrystals. The fact that $\deg \mathcal{G}>0$ combined with the particular choice of $b'$ implies that $\mathcal{G}\supseteq \tilde{\E}'$ or $\mathcal{G}\subseteq \tilde{\E}'$. Then $\mathcal{G}'\supseteq \E'$ or $\mathcal{G}'\subseteq \E'$. The later case is obviously impossible. For the former one, $\E''=\E_b/\mathcal{E}'$ must have a direct summand of slope $\frac{s}{r}$ which is also impossible as $\E''$ is semi-stable.
\end{proof}



\begin{thebibliography}{alpha}

\bibitem {BFH} C. Birkbeck, T. Feng, D. Hansen, S. Hong, Q. Li, A. Wang, L. Ye, \textsl{Extensions of vector bundles on the Fargues-Fontaine curve}. to appear in J. Inst. Math. Jussieu.

	
\bibitem{BH} I. Biswas, Y. I. Holla, \textsl{Harder-Narasimhan reduction of a principal bundle}. Nagoya Math. J., 174 :201-223,
2004.

\bibitem{Bo}N. Bourbaki, \textsl{Groupes et alg\`eres de Lie}, Chapitres 4, 5 et 6, Masson 1981.

\bibitem{BW}O. B\"ultel, T. Wedhorn, \textsl{Congruence relations for Shimura varieties associated to some unitary groups}, J. Inst. Math. Jussieu 5 (2006), 229-261.
\bibitem{CS}A. Caraiani, P. Scholze, \textsl{On the generic part of the cohomology of compact unitary Shimura varieties},
 Ann. Math.  186 (2017), no. 3, 649-766.
\bibitem{CFS}M. Chen, L. Fargues, X. Shen, \textsl{On the structure of some p-adic period domains}, arXiv:1710.06935, preprint, 2017.

\bibitem{CKV}M. Chen, M. Kisin, and E. Viehmann, \textsl{Connected components of affine Deligne-
	Lusztig varieties in mixed characteristic}, Compositio. Math. 151 (2015), 1697-1762.

\bibitem{CoFo} P. Colmez, J.-M. Fontaine, \textsl{Construction des représentations $p$-adiques semi-stables}, Invent. Math. 140 (2000), 1-43.

\bibitem{DOR}J.-F. Dat, S. Orlik, and M. Rapoport, \textsl{Period domains over finite and $p$-adic fields}, Cambridge Tracts in Mathematics, vol. 183, Cambridge University Press, Cambridge, 2010.

\bibitem{Dr} V. G. Drinfeld, \textsl{Coverings of $p$-adic symmetric regions}, Functional Anal. Appi.10 (1976), 29-40.

\bibitem{Fal}G. Faltings, \textsl{Coverings of $p$-adic period domains}, J. Reine Angew. Math. 643 (2010), 111-139.

\bibitem{F}L. Fargues, \textsl{Geometrization of the local Langlands correspondence: an overview}, arXiv:1602.00999, preprint, 2016.

\bibitem{F2}L. Fargues, \textsl{Quelques r\'esultats et conjectures concernant la courbe}, in ``De la g\'eom\'etrie alg\'ebrique aux formes
automorphes (I) - (Une collection d'articles en l'honneur du soixanti\`eme anniversaire de G\'erard Laumon)'', vol. 369, Ast\'erisque 2015.

\bibitem{F3}L. Fargues, \textsl{$G$-torseurs en th\'eorie de Hodge $p$-adique}, to appear in Compositio. Math.

\bibitem{F4}L. Fargues, \textsl{Lettre \`a Rapoport}.

\bibitem{FF}L. Fargues, J.-M. Fontaine, \textsl{Courbes et fibr\'es vectoriels en th\'eorie de Hodge $p$-adique},  Ast\'erisque 406, Soc. Math. France, 2018.

\bibitem{FS} L. Fargues, P. Scholze, \textsl{Geometrization of the local Langlands correspondence}, arXiv:2102.13459.

\bibitem{GoHeNi}U. G\"ortz, X. He, S. Nie, \textsl{Fully Hodge-Newton decomposable Shimura varieties}, arXiv: 1610.05381, to appear in Peking Mathematical Journal.

\bibitem{Har1}U. Hartl, \textsl{On period spaces for $p$-divisible groups}, C. R. Math. Acad. Sci. Paris
346 (2008), no. 21-22, 1123-1128.

\bibitem{Har} U. Hartl, \textsl{On a conjecture of Rapoport and Zink}, Invent. Math. 193 (2013), 627-696.



\bibitem{KL}K. S. Kedlaya, R. Liu, \textsl{Relative $p$-adic Hodge theory: Foundations}, Ast\'erisque 371, Soc. Math. France, 2015.
\bibitem{Ked} K. S. Kedlaya, \textsl{Sheaves, stacks, and shtukas},  Perfectoid spaces: Lectures from the 2017 Arizona Winter School, Mathematical Surveys and Monographs, 242, American Mathematical Society, 2019, 45-192.
\bibitem{Kot1}R.E. Kottwitz, \textsl{Isocrystals with additional structure}, Compositio Math., 56 (1985), no. 2, 201-220.
\bibitem{Kot2} R.E. Kottwitz, \textsl{Isocrystals with additional structure. II}, Compositio Math., 109 (1997), no. 3, 255-339.
\bibitem{Kot3}R.E. Kottwitz, \textsl{On the Hodge-Newton decomposition for split groups}, Int. Math. Res. Not., (2003), no. 26, 1433-1447.
\bibitem{LR} G. Laumon, M. Rapoport, \textsl{The Langlands lemma and the Betti numbers of stacks
of $G$-bundles on a curve}, Internat. J. Math., 7 (1996), no. 1, 29-45.

\bibitem{Man} E. Mantovan,
\textsl{ On non-basic Rapoport-Zink spaces}, Ann. Sci. \'Ec. Norm. Sup\'er., 41 (2008), no. 5, 671-716.

\bibitem{Ra}M. Rapoport, \textsl{Non-archimedean period domains}, In Proceedings of the International Congress of Mathematicians,
Vol. 1, 2 (Z\"urich, 1994), 423-434.

\bibitem{Ra2}M. Rapoport, \textsl{Accessible and weakly accessible period domains}, Appendix of \textsl{On the $p$-adic cohomology of the Lubin-Tate tower} by Scholze, Ann. Sci. Éc. Norm. Supér. 51 (2018), no. 4, 856-863.

\bibitem{RR}M. Rapoport, M. Richartz, \textsl{On the classification and specialization of $F$-isocrystals with additional structure}, Compositio Math. 103 (1996), no. 2, 153-181.

\bibitem{RZ}M. Rapoport, T. Zink, \textsl{Period spaces for $p$-divisible groups}, Ann. of Math. Stud. 141,
Princetion Univ. Press, 1996.

\bibitem{RV}M. Rapoport, E. Viehmann, \textsl{Towards a theory of local Shimura varieties}, M\"unster J. of Math. 7 (2014), 273-326.

\bibitem{Schi} S. Schieder, \textsl{The Harder-Narasimhan stratification of the moduli stack of
   $G$-bundles via Drinfeld's compactifications}, Selecta Math. (N.S.), volume 21  (2015), 763-831.
\bibitem{SW}P. Scholze, J. Weinstein, \textsl{Moduli of p-divisible groups},  Cambridge Journal of Mathematics 1 (2013), 145-237.

\bibitem{Sch}P. Scholze, J. Weinstein, \textsl{Berkeley lectures on $p$-adic geometry}, to appear in Annals of Math Studies.

\bibitem{Sch1}P. Scholze, \textsl{\'Etale cohomology of diamonds}, arXiv:1709.07343, preprint, 2018.

\bibitem{SerreCohoGal} J.-P. Serre, \textsl{Cohomologie Galoisienne}, Lecture Notes in Mathematics 5, cinqui{\`e}me \'edition, 1994, Springer.

\bibitem{Sh}X. Shen, \textsl{On the Hodge-Newton filtration for p-divisible groups with additional structures},  Int. Math. Res. Not. (2014), no. 13, 3582-3631.
\bibitem{Sh2}X. Shen, \textsl{Harder-Narasimhan strata and p-adic period domains}, arXiv:1909.02230.

\bibitem{Vi}E. Viehmann, \textsl{On Newton strata in the $B^+_{dR}$-Grassmannian}, arXiv:2101.07510.

\bibitem{W} J.-P. Wintenberger, \textsl{Propri\'et\'es du groupe tannakien des structures de Hodge $p$-adiques
et torseur entre cohomologies cristalline et \'etale}, Ann. Inst. Fourier (Grenoble) 47 (1997), 1289-1334.

\bibitem{Z}P. Ziegler, \textsl{Graded and filtered fiber functors on Tannakian categories}, J. Inst. Math. Jussieu, 14(2015), no. 1, 87-130.

\end{thebibliography}
\end{document}